\documentclass[a4paper,12pt]{amsart}

\usepackage{amsmath}
\usepackage{amssymb}
\usepackage{amsfonts}
\usepackage{graphicx}
\usepackage{mathtools}
\usepackage[colorlinks]{hyperref}
\renewcommand\eqref[1]{(\ref{#1})} %Need with hyperref

\graphicspath{ {images/} }
\newcommand{\jp}[1]{{\left\langle{#1}\right\rangle}}
\title[Heat and wave type equation on compact Lie Groups]{Heat and wave type equations with non-local operators, I. Compact Lie groups}

\author[W. A.A. de Moraes]{Wagner A.A. de Moraes}
\address{
	Wagner A.A. de Moraes:
	\endgraf
	Departamento de MatemÃ¡tica
	\endgraf
	Universidade Federal do ParanÃ¡, C.P. 19096, CEP 81531-990, Curitiba
	\endgraf
	Brazil
	\endgraf
	{\it E-mail address} {\rm wagnermoraes@ufpr.br}}

\author[J. E. Restrepo]{Joel E. Restrepo}
\address{
	Joel E. Restrepo:
	\endgraf
	Department of Mathematics: Analysis, Logic and Discrete Mathematics
	\endgraf
	Ghent University, Krijgslaan 281, Building S8, B 9000 Ghent
	\endgraf
	Belgium
	\endgraf
	{\it E-mail address} {\rm joel.restrepo@ugent.be; cocojoel89@yahoo.es}}

\author[M. Ruzhansky]{Michael Ruzhansky}
\address{
	Michael Ruzhansky:
	\endgraf
	Department of Mathematics: Analysis, Logic and Discrete Mathematics
	\endgraf
	Ghent University, Krijgslaan 281, Building S8, B 9000 Ghent
	\endgraf
	Belgium
	\endgraf
	and
	\endgraf
	School of Mathematical Sciences
		\endgraf Queen Mary University of London 
			\endgraf
		United Kingdom
			\endgraf
	{\it E-mail address} {\rm michael.ruzhansky@ugent.be}}


\newtheoremstyle{theorem}%name
{10pt}          % space above
{10pt}  % space below
{\sl}  % bofy font
{\parindent}     % ident - empty=no indent,  \parindent= paragraph indent
{\bf}  % thm head font
{. }    % punctuation after thm head
{ }    % space after thm head: `` ``=normal \newline=linebreak
{}     % thm head specification
\theoremstyle{theorem}
\newtheorem{theorem}{Theorem}

\numberwithin{equation}{section}
\theoremstyle{plain}

\theoremstyle{definition}

\newtheoremstyle{defi}%name
{10pt}          % space above
{10pt}  % space below
{\rm}  % bofy font
{\parindent}     % ident - empty=no indent,  \parindent= paragraph indent
{\bf}  % thm head font
{. }    % punctuation after thm head
{ }    % space after thm head: `` ``=normal \newline=linebreak
{}     % thm head specification
\theoremstyle{defi}
\newtheorem{definition}[theorem]{Definition}
\newtheorem{remark}[theorem]{Remark}

\def\proofname{\indent {\sl Proof.}}

\begin{document}
 	\begin{abstract}
We prove existence, uniqueness and give the analytical solution of heat and wave type equations on a compact Lie group $G$ by using a non-local (in time) differential operator and a positive left invariant operator (maybe unbounded) acting on the group. For heat type equations, solutions are given in $L^q(G)$ for data in $L^p(G)$ with $1<p\leqslant 2\leqslant q<+\infty$. We also provide some asymptotic estimates (large-time behavior) for the solutions. Some examples are given. Also, for wave type equations, we give the solution on some suitable Sobolev spaces over $L^2(G)$. We complement our results, by studying a multi-term heat type equation as well.        
	\end{abstract}
	\maketitle
	\tableofcontents

\section{Introduction}

	Studies of integro-differential equations with non-local operators have started long time ago. Specially, diffusion equations, wave equations and Cauchy problems were considered with the use of the Riemann-Liouville fractional integro-differential operators. We refer,  e.g. to the expository paper (223 pages) of M. Riesz about the integral of Riemann-Liouville for the Cauchy problem \cite{Riesz}. These ideas continued with people like Miller \cite{memory2}, Schneider and Wyss \cite{FractionalDiffusion}, who studied fractional diffusion and wave equations, see also \cite{memory1}. We also mention the papers published almost simultaneously (with respect to Schneider and Wyss) by Fujita in \cite{Fujita1,Fujita2}, where one presented a full investigation of integro-differential equations interpolating between heat and wave equations, using techniques which are different from previous ones. The latter ideas were mainly developed on the real axis.    
In the last decades, people have been interested to extend the previous works to $\mathbb{R}^n$. In this new setting, new approaches were developed and different equations of heat and wave type were studied by using time-variable non-local integro-differential operators of Riemann-Liouville. Researchers have found explicit solutions of the considered equations and have also answered questions about the well-posedness, regularity and large time behaviour of the solutions, we cite the following papers \cite{uno1,uno2,uno3,uno4,uno5} and references therein. Note that the idea of applying non-local operators in time variable has been also explored for other type of equations and problems. For instance, we have parabolic problems \cite{para}, Cauchy equations \cite{cauchy,Riesz}, fractional diffusion equations with bounded domains \cite{otherdomain}, initial boundary value problems and their applications \cite{initial}, Schauder estimates \cite{apli2},  continuous time random walks \cite{apli1}, just to mention a few of them, etc.        

In the last 40 years, the above classical and fundamental problems have turned out to be studied in some other more general scenarios like nilpotent Lie groups, where the harmonic and microlocal analysis of the group play an important role. For the heat equation on compact Lie groups see e.g. \cite{heat-compact2,heat-fundamental,heat-compact1}, while for the wave equation on such groups see \cite{wave-compact1,wave-compact2}. More general studies have been done on other Lie groups for the wave equation in \cite{homogeneous,palmieri}, and on the Heisenberg group in \cite{heisen}.

%and on compact Lie groups for left-invariant operators in \cite{KMP21}, \cite{KMR21}.        

In this paper, in Sections \ref{heat-section} and \ref{wave-section}, we study heat and wave type equations with a non-local (in time) differential operator of Caputo-type (so-called Dzhrbashian-Caputo fractional derivative), which allows one to interpolate between the classical heat and wave equations, on a compact Lie group $G$ along with a positive left invariant operator (maybe unbounded). The main result is about the $L^p(G)-L^q(G)$ estimates for the solution of the latter equation. We provide the explicit representation and the time decay rate for the solution. We illustrate the obtained results by some examples. In this part, we use some recent results on $L^p(G)-L^q(G)$ multipliers on locally compact groups \cite{RR2020}. This allows us to overcome and use the analysis on the group to just focus on the trace of the spectral projections of the considered left invariant operator. In the second part of the paper, in Section \ref{wave-section}, we study wave type equations. For this case, we can not use the latter approach. Nevertheless, we study the wave type equations in some suitable Sobolev spaces over $L^2(G)$ by using the Fourier analysis of the group. We find an explicit solution of the considered equation on $L^2(G)$. We complement the above results by studying a multi-term heat type equation, see Section \ref{multi-section}.    

Let us now give a brief sketch of the problems and their solutions, which will be resolved throughout the whole paper. In all results we have found the explicit representation of the solution of the considered equations. 

So, we start by studying the following heat type equation: 
\begin{equation}\label{Heat-intro}
	\left\{\begin{aligned}
		^{C}\partial_{t}^{\alpha}u(t,x)+\mathcal{L}u(t,x)&=0, \quad t>0,\,\, x\in G, \\
		u(t,x)|_{_{_{t=0}}}&=u_0(x),
	\end{aligned}
	\right.
\end{equation}
where $^{C}\partial_{t}^{\alpha}$ is the Dzhrbashyan-Caputo fractional derivative from \eqref{caputo-alternative-uso}, $G$ is a compact Lie group, $\mathcal{L}$ is a positive linear left invariant operator on $G$ (maybe unbounded) and $0<\alpha\leqslant1$.

The main result can be summarised as follows:

\begin{theorem}
	If $u_0\in L^p(G)$ for $1<p\leqslant 2$ and the condition \eqref{need} is satisfied then there exists a unique solution $u\in \mathcal{C}\big([0,+\infty); L^q(G)\big)$ for $2\leqslant q<+\infty$ of the Cauchy problem \eqref{Heat-intro}. In particular, if the condition \eqref{asymtotic-trace} holds then for any $1<p\leqslant 2\leqslant q<+\infty$ such that $\frac{1}{\lambda}>\frac{1}{p}-\frac{1}{q}$ we have the following time decay rate for the solution of equation \eqref{Heat-intro}: 
	\[
	\|u(t,\cdot)\|_{L^q(G)}\leqslant C_{\alpha,\lambda,p,q}t^{-\alpha\lambda\left(\frac{1}{p}-\frac{1}{q}\right)}\|u_0\|_{L^p(G)},    
	\]
	with the constant $C_{\alpha,\lambda,p,q}$ independent of $u_0$ and $t>0.$
\end{theorem}

The order of the time-decay above seems to be sharp. At least, in the case of $\mathbb{R}^n$, we can recover the sharp estimate given in \cite[Theorem 3.3, item (i)]{uno2} whenever $\frac{2}{n}>\frac{1}{p}-\frac{1}{q}$. The latter condition is the same given by $\frac{1}{\lambda}>\frac{1}{p}-\frac{1}{q}$ where $\lambda=n/2$, see \textit{Example 1} in Subsection \ref{examples}. 

For the sub-Laplacian on the group and on the torus, some explicit examples are given. More details can be found in Subsection \ref{examples}.  

We also study wave type equations. In fact, we focus on the following equation, which interpolates between wave (without being wave for $\alpha<2$) and heat types:  
\begin{equation}\label{wave-intro}
	\left\{ \begin{aligned}
		^{C}\partial_{t}^{\alpha}u(t,x)+\mathcal{L}u(t,x)&=0, \quad t>0,\,\, x\in G, \\
		u(t,x)|_{_{_{t=0}}}&=u_0(x), \\
		\partial_t u(t,x)|_{_{_{t=0}}}&=u_1(x),
	\end{aligned}
	\right.
\end{equation}
where $\mathcal{L}$ is a positive linear left invariant operator and $1<\alpha<2$. For this type of equations we use the Fourier analysis on the group to prove existence of a solution on a Sobolev space $\mathcal{H}_{\mathcal{L}}^{\beta}(G)$ $(\beta\in\mathbb{R})$ over $L^2(G)$ since we can not use the same approach as in the heat type equations. All details are provided in Section \ref{wave-section}. Thus, we will show that:

\begin{theorem}
Let $G$ be a compact Lie group, $1<\alpha<2$ and $\beta\in\mathbb{R}$. Let $\mathcal{L}$ be a positive linear left invariant operator on $G$. 
	\begin{enumerate}
		\item If $(u_0,u_1)\in \mathcal{H}_{\mathcal{L}}^{\beta}(G)\times \mathcal{H}_{\mathcal{L}}^{\beta}(G)$ then there exists a unique solution $u(t,\cdot)\in \mathcal{H}^{\beta+2}_{\mathcal{L}}(G)$ for any $t\in(0,+\infty)$ for the Cauchy problem \eqref{wave-intro} given explicitly by
		\begin{equation}\label{solution-wave-intro}
			u(t,x)=E_\alpha(-t^{\alpha}\mathcal{L})u_0(x)+tE_{\alpha,2}(-t^{\alpha}\mathcal{L})u_1(x),\quad x\in G,
		\end{equation}
		and we have 
\[
			\|u(t,\cdot)\|_{\mathcal{H}^{\beta+2}_{\mathcal{L}}(G)}\lesssim (1+t^{-\alpha})\|u_0\|_{\mathcal{H}_{\mathcal{L}}^{\beta}(G)}+t(1+t^{-\alpha})\|u_1\|_{\mathcal{H}_{\mathcal{L}}^{\beta}(G)}.
\]
		\item If $(u_0,u_1)\in \mathcal{H}^{\beta}_{\mathcal{L}}(G)\times \mathcal{H}^{\beta+2}_{\mathcal{L}}(G)$ then there exists a unique solution $u(t,\cdot)\in \mathcal{H}^{\beta+2}_{\mathcal{L}}(G)$ for any $t\in(0,+\infty)$ for the Cauchy problem \eqref{wave-intro} given explicitly by \eqref{solution-wave-intro}, and we have
\[
\|u(t,\cdot)\|_{\mathcal{H}^{\beta+2}_{\mathcal{L}}(G)}\lesssim (1+t^{-\alpha})\|u_0\|_{\mathcal{H}_{\mathcal{L}}^{\beta}(G)}+t\|u_1\|_{\mathcal{H}_{\mathcal{L}}^{\beta+2}(G)}.
			\]
	\item If $(u_0,u_1)\in \mathcal{H}^{\beta}_{\mathcal{L}}(G)\times \mathcal{H}_{\mathcal{L}}^{\beta+\frac{2(\alpha-1)}{\alpha}}(G)$ then there exists a unique solution $u(t,\cdot)\in \mathcal{H}^{\beta+2}_{\mathcal{L}}(G)$ for any $t\in(0,+\infty)$ for the Cauchy problem \eqref{wave-intro} given explicitly by \eqref{solution-wave-intro}, and we have
\[
\|u(t,\cdot)\|_{\mathcal{H}^{\beta+2}_{\mathcal{L}}(G)}\lesssim (1+t^{-\alpha})\|u_0\|_{\mathcal{H}_{\mathcal{L}}^{\beta}(G)}+t\|u_1\|_{\mathcal{H}_{\mathcal{L}}^{\beta}(G)}+\|u_1\|_{\mathcal{H}_{\mathcal{L}}^{\beta+\frac{2(\alpha-1)}{\alpha}}(G)}.
			\]
			
			\item If $(u_0,u_1)\in \mathcal{H}^{\beta+2}_{\mathcal{L}}(G)\times \mathcal{H}^{\beta}_{\mathcal{L}}(G)$ then there exists a unique solution $u(t,\cdot)\in \mathcal{H}^{\beta+2}_{\mathcal{L}}(G)$ for any $t\in(0,+\infty)$ for the Cauchy problem \eqref{wave-intro} given explicitly by \eqref{solution-wave-intro}, and we have
\[
\|u(t,\cdot)\|_{\mathcal{H}^{\beta+2}_{\mathcal{L}}(G)}\lesssim \|u_0\|_{\mathcal{H}_{\mathcal{L}}^{\beta+2}(G)}+t(1+t^{-\alpha})\|u_1\|_{\mathcal{H}_{\mathcal{L}}^{\beta}(G)}.
			\]
\item If $(u_0,u_1)\in \mathcal{H}^{\beta+2}_{\mathcal{L}}(G)\times \mathcal{H}^{\beta+2}_{\mathcal{L}}(G)$ then there exists a unique solution $u(t,\cdot)\in \mathcal{H}^{\beta+2}_{\mathcal{L}}(G)$ for any $t\in(0,+\infty)$ for the Cauchy problem \eqref{wave-intro} given explicitly by \eqref{solution-wave-intro}, and we have
\[
\|u(t,\cdot)\|_{\mathcal{H}^{\beta+2}_{\mathcal{L}}(G)}\lesssim \|u_0\|_{\mathcal{H}_{\mathcal{L}}^{\beta+2}(G)}+t\|u_1\|_{\mathcal{H}_{\mathcal{L}}^{\beta+2}(G)}.
	\]
	\item If $(u_0,u_1)\in \mathcal{H}^{\beta+2}_{\mathcal{L}}(G)\times \mathcal{H}_{\mathcal{L}}^{\beta+\frac{2(\alpha-1)}{\alpha}}(G)$ then there exists a unique solution $u(t,\cdot)\in \mathcal{H}^{\beta+2}_{\mathcal{L}}(G)$ for any $t\in(0,+\infty)$ for the Cauchy problem \eqref{wave-intro} given explicitly by \eqref{solution-wave-intro}, and we have
\[
\|u(t,\cdot)\|_{\mathcal{H}^{\beta+2}_{\mathcal{L}}(G)}\lesssim \|u_0\|_{\mathcal{H}_{\mathcal{L}}^{\beta+2}(G)}+t\|u_1\|_{\mathcal{H}_{\mathcal{L}}^{\beta}(G)}+\|u_1\|_{\mathcal{H}_{\mathcal{L}}^{\beta+\frac{2(\alpha-1)}{\alpha}}(G)}.
			\]
	\end{enumerate}
	Additionally, we also have that  \begin{equation*}
		\|\partial_{t}u(t,\cdot)\|_{\mathcal{H}^{\beta}_{\mathcal{L}}(G)}\lesssim \left\{
		\begin{array}{rccl}
			t^{-1}\|u_0\|_{\mathcal{H}^{\beta}_{\mathcal{L}}(G)}+\|u_1\|_{\mathcal{H}^{\beta}_{\mathcal{L}}(G)},& u_{0},u_1\in \mathcal{H}^{\beta}_{\mathcal{L}}(G),\\
			\|u_0\|_{\mathcal{H}^{\beta+2/\alpha}_{\mathcal{L}}(G)}+\|u_1\|_{\mathcal{H}^{\beta}_{\mathcal{L}}(G)},& u_0\in \mathcal{H}^{\beta+2/\alpha}_{\mathcal{L}}(G),\, u_1\in \mathcal{H}^{\beta}_{\mathcal{L}}(G),
		\end{array}
		\right.
	\end{equation*}
	for all $t\in(0,+\infty).$ The constants in the above estimates are independent of $t>0.$    
\end{theorem}

The above studies are complemented in Section \ref{multi-section} by studying the following multi-term heat type equations: 
\begin{equation}\label{multi-intro}
	\left\{ \begin{aligned}
		\prescript{C}{}\partial_{t}^{\alpha_0}u(t,x)+\gamma_1\prescript{C}{}\partial_{t}^{\alpha_1}u(t,x)+\cdots+\gamma_m\prescript{C}{}\partial_{t}^{\alpha_m}u(t,x)+\mathcal{L}u(t,x)&=0,\,\,  \\
		u(t,x)|_{_{_{t=0}}}&=u_0(x),
	\end{aligned}
	\right.
\end{equation}
for $t>0$ and $x\in G$, where $\mathcal{L}$ is a positive linear left invariant operator, $\gamma_i>0$ $(i=1,\ldots,m)$ and $0<\alpha_m<\alpha_{m-1}<\cdots<\alpha_1<\alpha_0\leqslant1$. So, we show that:
\begin{theorem}
	Let $G$ be a compact Lie group and $\beta\in\mathbb{R}$. Suppose also that $\mathcal{L}$ is a positive linear left invariant operator on $G$. 
	\begin{enumerate}
		\item If $u_0\in \mathcal{H}^{\beta}_{\mathcal{L}}(G)$ then there exists a unique solution $u(t,\cdot)\in \mathcal{H}^{\beta+2}_{\mathcal{L}}(G)$ for any $t\in(0,T]$ for the Cauchy problem \eqref{multi-intro} given explicitly by
		\begin{align}
			&u(t,x)= \nonumber\\
			&\sum_{k=0}^{m}t^{\alpha_0-\alpha_k}E_{(\alpha_0-\alpha_1,\ldots,\alpha_0-\alpha_m,\alpha_0),\alpha_0-\alpha_k+1}(-\gamma_1 t^{\alpha_0-\alpha_1},\ldots,-\gamma_m t^{\alpha_0-\alpha_m},-t^{\alpha_0}\mathcal{L}) u_0(x),\label{multi-solution-intro}
		\end{align}
		and we have 
		\[
		\|u(t,\cdot)\|_{\mathcal{H}_{\mathcal{L}}^{\beta+2}(G)}\leqslant C_{T,\alpha_0,\ldots,\alpha_m}\left(\sum_{k=0}^{m}\gamma_k t^{\alpha_0-\alpha_k}\right)(1+t^{-\alpha_0})\|u_0\|_{\mathcal{H}^{\beta}_{\mathcal{L}}(G)}. 
		\]
		\item If $u_0\in \mathcal{H}_{\mathcal{L}}^{\beta+2}(G)$ then there exists a unique solution $u(t,\cdot)\in \mathcal{H}^{\beta+2}_{\mathcal{L}}(G)$ for any $t\in(0,T]$ for the Cauchy problem \eqref{multi-intro} given explicitly by \eqref{multi-solution-intro}, and we have 
		\[
		\|u(t,\cdot)\|_{\mathcal{H}_{\mathcal{L}}^{\beta+2}(G)}\leqslant C_{T,\alpha_0,\ldots,\alpha_m}\left(\sum_{k=0}^{m}\gamma_k t^{\alpha_0-\alpha_k}\right)\|u_0\|_{\mathcal{H}^{\beta+2}_{\mathcal{L}}(G)}. 
		\] 
	\end{enumerate}
\end{theorem}

\section{Preliminary Results}

In this section we collect definitions and results on non-local (in time) differential operators and compact Lie groups, which will be used throughout the whole paper. 

\subsection{Non-local (in time) differential operators}

We start by giving some definitions of several basic function spaces. For a fixed finite interval $[a,T]\subseteq\mathbb{R}$, we recall the following well-known function spaces on this interval:
\begin{align*}
	L^1(a,T)&=\left\{f:(a,T)\to\mathbb{C}\text{ measurable}\;:\;\big\|f\big\|_{L^1(a,T)}:=\int_a^T\big|f(t)\big|\,\mathrm{d}t<+\infty\right\}; \\
	AC[a,T]&=\left\{f:[a,T]\to\mathbb{C}\;:\;f\text{ absolutely continuous on }[a,T]\right\}; \\
	AC^n[a,T]&=\left\{f:[a,T]\to\mathbb{C}\;:\;f^{(n-1)}\text{ exists and is in }AC[a,T]\right\},\quad n\in\mathbb{N}.
\end{align*}
%We know that $L(a,T)$ is a Banach space under its norm $\|\cdot\|_{L(a,T)}$, and that $AC[a,T]$ is a Banach space under its norm $\|\cdot\|_{AC[a,T]}.$ 
Now we recall the Riemann--Liouville fractional integral of order $\beta>0$ (\cite{kilbas,samko}), which is defined as follows \cite[Sections 2.3 and 2.4]{samko}:
\[
\prescript{RL}{a}I^{\beta}f(t)=\frac1{\Gamma(\beta)}\int_a^t (t-s)^{\beta-1}f(s)\,\mathrm{d}s,\qquad f\in L^1(a,T).
\]
In this paper, we use a non-local differential operator in time with memory kernel $\frac{t^{\alpha-1}}{\Gamma(\alpha)}$, which is commonly known as the Dzhrbashyan--Caputo fractional derivative, defined by
\begin{equation}\label{caputo-alternative-uso}
	\prescript{C}{a}D^{\beta}f(t)=\prescript{RL}{a}I^{n-\beta}f^{(n)}(t),\qquad f\in AC^n[a,T],\quad n:=\lfloor\beta\rfloor+1.
\end{equation}
The fundamental reason for considering this kernel is that it allows the considered equation in this paper to interpolate between the heat equation $(\alpha=1)$ and the wave equation $(\alpha=2)$. It is frequently used in applications since it symbolizes the memory of a long-time tail of the power order \cite{symbol}.

The above operator can be equivalently expressed as 
\begin{equation}\label{caputo-alternative}
	\prescript{C}{a}D^{\beta}f(t)=\prescript{RL}{a}D^{\beta}\left(f(t)-\sum_{k=0}^{n-1}\frac{f^{(k)}(a)}{k!}(t-a)^k\right),\qquad f\in AC^n[a,T],
\end{equation}
where $\prescript{RL}{a}D^{\beta}f(t)$ is the Riemann--Liouville fractional derivative of order $\beta\geqslant0$ given by
\begin{align*}
	\prescript{RL}{a}D^{\beta}f(t)&=D^{n}\prescript{RL}{a}I^{n-\beta}f(t) \\
	&=\frac1{\Gamma(n-\beta)}\left(\frac{\rm{d}}{\rm{d}t}\right)^{n}\int_a^t (t-s)^{n-\beta-1}f(s)\,\mathrm{d}s,\,\, f\in AC^n[a,T],\,\, n=\lfloor\beta\rfloor+1.
\end{align*}
Both definitions coincide for any function $f\in AC^n[a,T]$,  \cite[Theorem 3.1]{diethelm}, see also \cite[Theorem 2.2]{samko}. The only difference between the two definitions is the possibility of defining \eqref{caputo-alternative} on a larger function space than $AC^n[a,T]$, since it is possible to define the Riemann--Liouville derivative on such larger function spaces. A concrete example can be found in \cite{example}, where some functions are examined which have no first order derivative but have Riemann--Liouville fractional derivatives of all orders less than one.

\subsection{Compact Lie groups}
The majority of the notations and preliminary results important for the development of this investigation are recalled in this section. The references \cite{FR16} and \cite{livropseudo} provide a thorough exposition of these principles as well as demonstrations of all of the conclusions described here. For more classical books see e.g. \cite{compact1,compact2}. 

Let $G$ be a compact Lie group, and let the set of continuous irreducible unitary representations of $G$ be $\operatorname{Rep}(G)$. Every continuous irreducible unitary representation $\xi$ is finite dimensional since $G$ is compact, and it can be seen as a matrix-valued function $\xi:G \to \mathbb{C}^{d_\xi\times d_\xi}$, where $d_\xi = \dim \xi$. We say that $\xi \sim \psi$ if there exists an unitary matrix $A\in C^{d_\xi \times d_\xi}$ such that $A\xi(x) =\psi(x)A$, for all $x\in G$. The quotient of $\operatorname{Rep}(G)$ by this equivalence relation will be denoted by $\widehat{G}$, the unitary dual.

As usual, we denote by $L^{p}(G)$ $(1\leqslant p<+\infty)$ the space of $p$-integrable functions in $G$ with respect to the Haar measure (normalized) and essentially bounded for $p=+\infty$.

For $f \in L^1(G)$ the group Fourier transform of $f$ at $\xi \in \operatorname{Rep}(G)$ is
\begin{equation*}
	\widehat{f}(\xi)=\int_G f(x) \xi(x)^* \, \mathrm{d}x,
\end{equation*}
where $\mathrm{d}x$ is the normalized Haar measure on $G$. Precisely, the components of the matrix $\widehat{f}(\xi)$ are given by
$$
\widehat{f}(\xi)_{ij} = \int_G f(x) \overline{\xi(x)_{ji}} \mathrm{d} x,
$$
for every $1 \leqslant i,j \leqslant d_\xi$.
By the Peter-Weyl theorem, we have that 
\begin{equation}\label{ortho}
	\mathcal{B} := \left\{\sqrt{d_\xi} \, \xi_{ij} \,; \ \xi=(\xi_{ij})_{i,j=1}^{d_\xi}, [\xi] \in \widehat{G} \right\},
\end{equation}
is an orthonormal basis for $L^2(G)$, where we consider only one matrix unitary representation in each class of equivalence, and we may write
\begin{equation*}
	f(x)=\sum_{[\xi]\in \widehat{G}}d_\xi \operatorname{Tr}(\xi(x)\widehat{f}(\xi)).
\end{equation*}
Moreover, the Plancherel formula holds:
\begin{equation}
	\label{plancherel} \|f\|_{L^{2}(G)}=\left(\sum_{[\xi] \in \widehat{G}}  d_\xi \ 
	\|\widehat{f}(\xi)\|_{{\mathtt{HS}}}^{2}\right)^{\tfrac{1}{2}}=:
	\|\widehat{f}\|_{\ell^{2}(\widehat{G})},
\end{equation}
where 
\begin{equation*} \|\widehat{f}(\xi)\|_{{\mathtt{HS}}}^{2}=\operatorname{Tr}(\widehat{f}(\xi)\widehat{f}(\xi)^{*})=\sum_{i,j=1}^{d_\xi}  \bigr|\widehat{f}(\xi)_{ij}\bigr|^2.
\end{equation*}
Let $\mathcal{L}_G$ be the Laplace-Beltrami operator of $G$. For each $[\xi] \in \widehat{G}$, its matrix elements are eigenfunctions of $-\mathcal{L}_G$ corresponding to the same eigenvalue that we will denote by $\lambda_\xi^2$, where $\lambda_\xi^2 \geqslant 0$. Thus
\begin{equation}\label{symbol-laplace}
	-\mathcal{L}_G \xi_{ij}(x) = \lambda_\xi^2\xi_{ij}(x), \quad \textrm{for all } 1 \leqslant i,j \leqslant d_\xi.
\end{equation}
The  symbol of a continuous linear operator $P$ in $x\in G$ and $\xi \in \mbox{{Rep}}(G)$, $\xi=(\xi_{ij})_{i,j=1}^{d_\xi}$ is defined as
$$
\sigma_P(x,\xi) := \xi(x)^*(P\xi)(x) \in \mathbb{C}^{d_\xi \times d_\xi},
$$
where $(P\xi)(x)_{ij}:= (P\xi_{ij})(x)$, for all $1\leqslant i,j \leqslant d_\xi$, and we have
$$
Pf(x) = \sum_{[\xi] \in \widehat{G}} d_\xi \mbox{Tr} \left(\xi(x)\sigma_P(x,\xi)\widehat{f}(\xi)\right)
$$
for every $f \in C^\infty(G)$ and $x\in G$.

Notice that the last expression is independent of the choice of the representative. When $P: C^\infty(G) \to C^\infty(G)$ is a continuous linear left invariant operator, that is $P\pi_L(y)=\pi_L(y)P$, for all $y\in G$, we have that $\sigma_P$ is independent of $x\in G$ and
$$
\widehat{Pf}(\xi) = \sigma_P(\xi)\widehat{f}(\xi),
$$
for all $f \in C^\infty(G)$ and $[\xi] \in \widehat{G}$. For instance, the Laplace-Beltrami operator $\mathcal{L}_G$ is a left invariant operator and by \eqref{symbol-laplace} its symbol is $\sigma_{-\mathcal{L}_G}(\xi) = \lambda_\xi^2 \operatorname{Id}_{d_\xi \times d_\xi}$, for every $[\xi] \in \widehat{G}$. So, we have that
$$
\widehat{-\mathcal{L}_Gf}(\xi) = \lambda_\xi^2\widehat{f}(\xi).
$$

We denote by $\mathcal{M}(\widehat{G})$ the space consisting of all mappings
$$
F: \widehat{G} \rightarrow \bigcup_{[\xi] \in \widehat{G}} \mathcal{L}\left(\mathcal{H}_{\xi}\right) \subset \bigcup_{m=1}^{\infty} \mathbb{C}^{m \times m}
$$
satisfying $F([\xi]) \in \mathcal{L}\left(\mathcal{H}_{\xi}\right)$ for every $[\xi] \in \widehat{G}$. In matrix representations, we can view $F([\xi])$ as a matrix in $\mathbb{C}^{d_\xi \times d_\xi}$. In order to simplify the notation we will write $F(\xi)$ with a convention that $F \in \mathcal{M}(\widehat{G})$ if $F(\xi)=F(\eta)$ whenever $\xi \sim \eta$.

The space $\mathcal{S}^{\prime}(\widehat{G})$ of slowly increasing or tempered distributions on the unitary dual $\widehat{G}$ is defined as the space of all $H \in \mathcal{M}(\widehat{G})$ for which there exists some $k \in \mathbb{N}$ such that
$$
\sum_{[\xi] \in \widehat{G}} d_{\xi}\langle\xi\rangle^{-k}\|H(\xi)\|_{{\mathtt{HS}}}<+\infty,
$$
where $\jp{\xi} := \sqrt{1+\lambda_\xi^2}$.

The convergence in $\mathcal{S}^{\prime}(\widehat{G})$ is defined as follows. We will say that $H_{j} \in \mathcal{S}^{\prime}(\widehat{G})$ converges to $H \in \mathcal{S}^{\prime}(\widehat{G})$ in $\mathcal{S}^{\prime}(\widehat{G})$ as $j \rightarrow \infty$, if there exists some $k \in \mathbb{N}$ such that
$$
\sum_{[\xi] \in \widehat{G}} d_{\xi}\langle\xi\rangle^{-k}\left\|H_{j}(\xi)-H(\xi)\right\|_{{\mathtt{HS}}} \rightarrow 0
$$
as $j \rightarrow \infty$.

For $1 \leqslant p < +\infty$, we define the space $L^p(\widehat{G})$ as the space of all $H\in \mathcal{S}'(\widehat{G})$ such that 
$$
\|H\|_{L^{p}(\widehat{G})}:=\left(\sum_{[\xi] \in \widehat{G}}d_{\xi}^{p\left(\frac{2}{p}-\frac{1}{2}\right)}\|H(\xi)\|_{{\mathtt{HS}}}^{p}\right)^{1 / p}<+\infty.
$$
We refer to \cite{livropseudo} for the extensive analysis of this family of spaces. For $p=+\infty$, the space $L^\infty(\widehat{G})$ consists of all $H\in \mathcal{S}'(\widehat{G})$ such that 
$$
\|H\|_{L^{\infty}(\widehat{G})}:=\sup _{[\xi] \in \widehat{G}}d_{\xi}^{-1 / 2}\|H(\xi)\|_{{\mathtt{HS}}}<+\infty.
$$

The spaces $L^p(\widehat{G})$ are Banach spaces for all $1 \leqslant p \leqslant +\infty$ and for the special case $p=2$ we have that the space $L^2(\widehat{G})$ is a Hilbert space with the inner product 
$$
(E, F)_{L^{2}(\widehat{G})}:=\sum_{[\xi] \in \widehat{G}} d_{\xi} \operatorname{Tr}\left(E(\xi) F(\xi)^{*}\right).
$$
Moreover, the Fourier transform $f \mapsto \mathcal{F}_Gf := \widehat{f}$ defines a surjective isometry $L^2(G) \to L^2(\widehat{G})$ and the inverse Fourier transform is given by
$$
\left(\mathcal{F}_{G}^{-1} H\right)(x):=\sum_{[\xi] \in \widehat{G}} d_{\xi} \operatorname{Tr}(\xi(x) H(\xi))
$$
and we have
$$
\mathcal{F}_{G}^{-1} \circ \mathcal{F}_{G}=\operatorname{Id} \quad \text { and } \quad \mathcal{F}_{G} \circ \mathcal{F}_{G}^{-1}=\operatorname{Id}
$$
on $L^{2}(G)$ and $L^{2}(\widehat{G})$, respectively. Moreover, the Fourier transform $\mathcal{F}_{G}$ is unitary.

We also have that the Fourier transform $\mathcal{F}_{G}$ is a linear bounded operator from $L^{1}(G)$ to $L^{\infty}(\widehat{G})$ satisfying \cite[Prop. 10.3.42]{livropseudo}
\begin{equation}\label{l1-l2}
	\|\widehat{f}\|_{L^{\infty}(\widehat{G})} \leqslant\|f\|_{L^{1}(G)},
\end{equation}
and the inverse Fourier transform $\mathcal{F}_{G}^{-1}$ is a linear bounded operator from $L^{1}(\widehat{G})$ to $L^{\infty}(G)$ satisfying
$$
\left\|\mathcal{F}_{G}^{-1} H\right\|_{L^{\infty}(G)} \leqslant\|H\|_{L^{1}(\widehat{G})} .
$$

For $s\in\mathbb{R}$, we may characterize the Sobolev space $H^s(G)$ as
$$
H^{s}(G)=\left\{f \in \mathcal{D}^{\prime}(G):\langle\xi\rangle^{s} \widehat{f}(\xi) \in L^{2}(\widehat{G})\right\}.
$$

For $k \in \mathbb{N}$, the Fourier transform $\mathcal{F}_G$ is a continuous bijection from $H^{2k}(G)$ to the space
$$
\left\{F \in \mathcal{M}(\widehat{G}): \sum_{[\xi] \in \widehat{G}} d_{\xi}\langle\xi\rangle^{2 k}\|F(\xi)\|_{{\mathtt{HS}}}^{2}<+\infty\right\}.
$$

	\section{Heat type equations with non-local differential  operators}\label{heat-section}

In this section we study the following heat type equation:  

\begin{equation}\label{HeatTypeEquationG}
	\left\{ \begin{aligned}
		^{C}\partial_{t}^{\alpha}u(t,x)+\mathcal{L}u(t,x)&=0, \quad t>0,\,\, x\in G, \\
		u(t,x)|_{_{_{t=0}}}&=u_0(x),
	\end{aligned}
	\right.
\end{equation}
where $^{C}\partial_{t}^{\alpha}$ is the Dzhrbashyan-Caputo fractional derivative from \eqref{caputo-alternative-uso}, $G$ is a compact Lie group, $\mathcal{L}$ is a positive linear left invariant operator on $G$ (we always assume densely defined and maybe unbounded) and $0<\alpha\leqslant1$. The case $\alpha=1$ coincides with the classical time derivative and heat equation. Therefore, in the following proofs we just focus on the case of $0<\alpha<1$.  

Here we show that there exists a unique continuous solution of equation \eqref{HeatTypeEquationG} which satisfies the $L^p(G)-L^q(G)$ estimates for $1\leqslant p\leqslant 2\leqslant q<+\infty$. In fact, we will see that the $L^p(G)-L^q(G)$ properties can be reduced to the time asymptotics of its propagator in the noncommutative Lorentz space norm \cite[Def. 2.12]{RR2020}. For the latter fact, we just need to impose a condition over the behaviour of the heat-propagator which involves calculating the trace of the spectral projections of the operator $\mathcal{L}$. We provide some examples to show the nature (viability) of this condition. For details on spectral theory, see e.g. \cite{BorelFunctional}. In the second part, we give the time decay rate of the propagator.

For the next result and sections we need to recall the two-parametric Mittag-Leffler function
\begin{equation}\label{bimittag}
	E_{\alpha,\rho}(z)=\sum_{k=0}^{+\infty} \frac{z^k}{\Gamma(\alpha k+\rho)},\quad z,\rho\in\mathbb{C},\quad \Re(\alpha)>0,
\end{equation}
which is absolutely and locally uniformly convergent for the given parameters. For more details of this function and some other types, we recommend the expository book \cite{mittag}. 

We also need to remember the trace of the spectral projections of a positive linear left invariant operator $\mathcal{L}$ acting on a compact Lie group, which is denoted by $\tau$, and given by
\[
\tau(E_{(0,s)}(\mathcal{L}))=\sum_{\xi\in\widehat{G}}d_\xi \sum_{k=1,\ldots,d_\xi,\,\,s_{k,\xi}<s}1,
\]
where each $s_{k,\xi}$ is a joint eigenvalue of $\mathcal{L}$ with the eigenfunction $\xi_{jk}$ for $j=1,\ldots,d_{\xi}$. A proof of this can be found in Subsection 7.2 (exactly on page 44) of \cite{RR2020preprint}. We point out that the preprint \cite{RR2020preprint} has been included and cited here since it contains additional and useful information (e.g. the above result on the trace) which can not be found in the published version \cite{RR2020}.      
\begin{theorem}\label{Main-heat}
	Let $G$ be a compact Lie group, $0<\alpha\leqslant1$ and $1\leqslant p\leqslant 2\leqslant q<+\infty$. Let $\mathcal{L}$ be a positive linear left invariant operator on $G$ (maybe unbounded) such that 
	\begin{equation}\label{need}
		\sup_{t>0}\sup_{s>0}[\tau\big(E_{(0,s)}(\mathcal{L})\big)]^{\frac{1}{p}-\frac{1}{q}}E_\alpha(-t^{\alpha}s)<+\infty.   
	\end{equation}
	If $u_0\in L^p(G)$ then there exists a unique solution $u\in \mathcal{C}\big([0,+\infty); L^q(G)\big)$ for the Cauchy problem \eqref{HeatTypeEquationG} given explicitly by
	\[
	u(t,x)=E_\alpha(-t^{\alpha}\mathcal{L})u_0(x),\quad t>0,\,\,x\in G,
	\]
	where the propagator is defined as
	\[
	E_\alpha(-t^{\alpha}\mathcal{L}) = \sum_{k=0}^{+\infty} \frac{(-t^{\alpha}\mathcal{L})^k}{\Gamma(\alpha k+1)}.
	\]
	In particular, if for some $\lambda>0$ we have 
	\begin{equation}\label{asymtotic-trace}
		\tau\big(E_{(0,s)}(\mathcal{L})\big)\lesssim s^{\lambda},\quad s\to+\infty,
	\end{equation}
	then \eqref{need} is satisfied for any $1<p\leqslant 2\leqslant q<+\infty$ such that $\frac{1}{\lambda}>\frac{1}{p}-\frac{1}{q}$, and one has the following time decay rate for the solution of equation \eqref{HeatTypeEquationG}: 
	\[
	\|u(t,\cdot)\|_{L^q(G)}\leqslant C_{\alpha,\lambda,p,q}t^{-\alpha\lambda\left(\frac{1}{p}-\frac{1}{q}\right)}\|u_0\|_{L^p(G)},    
	\]
	with the constant $C_{\alpha,\lambda,p,q}$ independent of $u_0$ and $t>0.$
\end{theorem}
\begin{proof}
	Let $[\xi]\in \widehat{G}$, and denote by $\widehat{u}(t,\xi)$ the group Fourier transform of $u$ with respect to the variable $x$. Therefore, from equation \eqref{HeatTypeEquationG}, we obtain
	\[
	\left\{ \begin{aligned}
		^{C}\partial_{t}^{\alpha}\widehat{u}(t,\xi)+\sigma_{\mathcal{L}}(\xi)\widehat{u}(t,\xi)&=0, \quad t>0, \\
		\widehat{u}(t,\xi)|_{_{_{t=0}}}&=\widehat{u_0}(\xi).
	\end{aligned}
	\right.
	\]
	Note that for the left invariance of operator $\mathcal{L}$, we have that $\sigma_{\mathcal{L}}$ is independent of $x\in G$ and $\widehat{\mathcal{L}f}(\xi) = \sigma_{\mathcal{L}}(\xi)\widehat{f}(\xi),$ for all $f\in C^\infty(G)$ and $[\xi] \in \widehat{G}$, see e.g. \cite{livropseudo}. Therefore, the matrix $\sigma_\mathcal{L}(\xi)$ can be written as
	\[
	\sigma_\mathcal{L}(\xi)=
	\begin{pmatrix}
		\mu_{1,\xi} & 0 &\ldots & 0\\
		0 & \mu_{2,\xi} & \ldots & 0 \\
		\vdots & \vdots &  & 0 \\
		0 & \ldots &  & \mu_{d_\xi,\xi}
	\end{pmatrix}
	\]
	where all the $\mu_i's$ are non-negative since the operator is positive.    
	
	Here we get a system of scalar ODEs, which depend on the dimension $d_\xi$ of the representation $[\xi]$. In fact, we have
	\[
	\left\{ \begin{aligned}
		^{C}\partial_{t}^{\alpha}\widehat{u}(t,\xi)_{ij}+\mu_{i,\xi}\widehat{u}(t,\xi)_{ij}&=0,\quad t>0, \\
		\widehat{u}(t,\xi)_{ij}|_{_{_{t=0}}}&=\widehat{u_0}(\xi)_{ij},
	\end{aligned}
	\right.
	\]
	for any $i,j\in \{1,\ldots,d_\xi\}$ with the non-negative eigenvalues $\mu_{i,\xi}$. Let us now apply the Laplace transform in the time-variable, and hence
	\[
	\left\{ \begin{aligned}
		s^{\alpha}\widetilde{\widehat{u}}(s,\xi)_{ij}-s^{\alpha-1}\widehat{u_0}(\xi)_{ij}+\mu_{i,\xi} \widetilde{\widehat{u}}(s,\xi)_{ij}&=0, \quad s>0, \\
		\widehat{u}(t,\xi)_{ij}|_{_{_{t=0}}}&=\widehat{u_0}(\xi)_{ij}.
	\end{aligned}
	\right.
	\]
	Thus, it follows that
	\[
	\widetilde{\widehat{u}}(s,\xi)_{ij}=\frac{s^{\alpha-1}}{s^{\alpha}+\mu_{i,\xi}}\widehat{u_0}(\xi)_{ij}, \quad s>0, 
	\]
	and by the application of the inverse Laplace transform (see e.g.  \cite[Theorem 2.1]{new-mittag-add}) we arrive at
	\[
	\widehat{u}(t,\xi)_{ij}=E_{\alpha}(-\mu_{i,\xi} t^{\alpha})\widehat{u_0}(\xi)_{ij}.
	\]
	Now we apply the inverse Fourier transform in $x$ and get the explicit solution of the considered problem as follows: 
	\begin{align*}
		u(t,x)&=\sum_{[\xi]\in\widehat{G}}d_{\xi}\sum_{i,j=1}^{d_{\xi}}\widehat{u}(t,\xi)_{ij}\xi(x)_{ji}=\sum_{[\xi]\in\widehat{G}}d_{\xi}\sum_{i,j=1}^{d_{\xi}}E_{\alpha}(-\mu_{i,\xi} t^{\alpha})\widehat{u_0}(\xi)_{ij}\xi(x)_{ji}\\
		&=\sum_{k=0}^{+\infty}\frac{t^{\alpha k}}{\Gamma(\alpha k+1)}\sum_{[\xi]\in\widehat{G}}d_{\xi}\sum_{i,j=1}^{d_{\xi}}(-\mu_{i,\xi})^{k}\widehat{u_0}(\xi)_{ij}\xi(x)_{ji} \\
		&=\sum_{k=0}^{+\infty}\frac{(-t^{\alpha})^ k}{\Gamma(\alpha k+1)}\sum_{[\xi]\in\widehat{G}}d_{\xi}\sum_{i,j=1}^{d_{\xi}}\widehat{\mathcal{L}^{k}u_0}(\xi)_{ij}\xi(x)_{ji} \\
		&=\sum_{k=0}^{+\infty}\frac{(-t^{\alpha})^ k}{\Gamma(\alpha k+1)}\mathcal{L}^{k}u_0(x)=E_\alpha(-t^\alpha \mathcal{L})u_0(x). 
	\end{align*}
	Thus, by \eqref{need} and \cite[Corollary 6.2]{RR2020} we get that 
	\[
	\|E_\alpha(-t^{\alpha}\mathcal{L})\|_{L^{p}(G)\to L^q(G)}<+\infty,
	\]
	which completes the first part of the proof. Notice that our left invariant operator $\mathcal{L}$ in $G$ is a Fourier multiplier on $G$, see \cite[Remarks 2.17 and 5.8]{RR2020}. 
	
	On the other hand, by \cite[Theorem 5.1]{RR2020} we have
	\begin{equation}\label{previous}
		\|u(t,\cdot)\|_{L^q(G)}=\|E_\alpha(-t^{\alpha}\mathcal{L})u_0(\cdot)\|_{L^q(G)}\lesssim \|E_\alpha(-t^\alpha \mathcal{L})\|_{L^{r,\infty}(VN_R(G))}\|u_0\|_{L^p(G)},
	\end{equation}
	where the above Lorentzian norm is given by \cite[Theorem 6.1]{RR2020}:
	\[
	\|E_\alpha(-t^\alpha \mathcal{L})\|_{L^{r,\infty}(VN_R(G))}=\sup_{s>0}[\tau\big(E_{(0,s)}(\mathcal{L})\big)]^{\frac{1}{r}}E_\alpha(-t^{\alpha}s),\quad \frac{1}{r}=\frac{1}{p}
	-\frac{1}{q}, 
	\]
	where the group von Neumann algebra $VN_R(G)$ is generated by all the right actions of $G$ on $L^2(G)$ ($\pi_R(g)f(x)=f(xg)$ with $g\in G$), which means that $VN_R(G)=\{\pi_R(G)\}^{!!}_{g\in G}$, where $!!$ is the bicommutant of the self-adjoint subalgrabras $\{\pi_R(g)\}_{g\in G}\subset \mathcal{L}(L^2(G))$. The latter result is a consequence of the fact \cite{von}: $VN_R(G)^{!}=VN_L(G)$ and $VN_L(G)^{!}=VN_R(G)$, where the symbol $!$ represents the commutant of the group von Neumann algebra. For the above result we need the operator $\mathcal{L}$ to be affiliated with the semifinite von Neumann algebra $VN_R(G)$, which is provided by the left invariance \cite[Remark 2.17]{RR2020}. We also used that $E_\alpha(-t)$, $t\geqslant0$, is completely monotonic \cite{Pollard} such that $E_\alpha(0)=1$ and $\displaystyle\lim_{t\to+\infty}E_\alpha(-t)=0$ by the uniform estimate given in \cite[Theorem 4]{Mittag-bounded}. 
	
	So, by using the hypothesis \eqref{asymtotic-trace} and again  \cite[Theorem 4]{Mittag-bounded} we get
	\begin{equation}\label{sup-propa}
		\|E_\alpha(-t^\alpha \mathcal{L})\|_{L^{r,\infty}(VN_R(G))}\lesssim \sup_{s>0}s^{\frac{\lambda}{r}}\frac{1}{1+\frac{t^\alpha s}{\Gamma(1+\alpha)}}.    
	\end{equation}
	Let us now see that the above supremum is attained at $s=\frac{\lambda \Gamma(1+\alpha)}{r\left(1-\frac{\lambda}{r}\right)}t^{-\alpha}.$ In fact, take 
	\[
	g(s)=\frac{s^{\frac{\lambda}{r}}}{\Gamma(1+\alpha)+t^\alpha s},\quad s>0.
	\]
	We calculate its derivative 
	\[
	g^\prime(s)=\frac{s^{\lambda/r}\big(\frac{\lambda}{r}\Gamma(1+\alpha)s^{-1}+\frac{\lambda}{r}t^{\alpha}-t^\alpha\big)}{(\Gamma(1+\alpha)+t^\alpha s)^2}.
	\]
	So, the only zero for $s>0$ is at $s^*=\frac{\lambda \Gamma(1+\alpha)}{r\left(1-\frac{\lambda}{r}\right)}t^{-\alpha}$, which is conditioned to $\frac{1}{\lambda}>\frac{1}{p}-\frac{1}{q}$ (this provides the positiveness of the point $s$). It can be also inferred that the function $g^\prime(s)$ changes its sign from positive to negative at the point $s^*$. Therefore, $s^*$ is a point of maximum of the function $g(s)$. By \eqref{sup-propa}, we have
	\begin{equation}\label{asymtotic-heat}
		\|E_\alpha(-t^\alpha \mathcal{L})\|_{L^{r,\infty}(VN_R(G))}\lesssim \sup_{s>0}\frac{s^{\frac{\lambda}{r}}}{\Gamma(1+\alpha)+t^\alpha s}\leqslant C_{\alpha,\lambda,p,q}t^{-\alpha \lambda/r},
	\end{equation}
	and then the result follows immediately by \eqref{previous}.
	
\end{proof}
\begin{remark}
	Notice that the Mittag-Leffler function of negative argument $E_\alpha(-t^{\alpha}x)$ $(t,x>0)$ is completely monotonic for all $0<\alpha\leqslant1$ \cite{mittagalpha,Pollard}. We also have that our linear closed operator $\mathcal{L}$ in a Hilbert space $\mathcal{H}$ (we have this from the left invariant property on $G$, see \cite[Remark 2.17]{RR2020} and \cite[Definition 2.1]{RR2020}), is a sectorial operator of angle $0$, or what is the same a positive sectorial operator \cite[Chapter 2]{functionalcalculus}. So, by using the Borel functional calculus we can give sense to the propagator $E_\alpha(-t^{\alpha}\mathcal{L}).$ In Theorem \ref{Main-heat} we found the solution by using the Fourier analysis on the group. Nevertheless, in \cite[Chapter 3]{thesis}, we can see that the global mild solution (\cite[Def. 3.1]{thesis}) of equation \eqref{HeatTypeEquationG} (found in a Banach space) is expressed in a different form. In fact, we can represent the propagator as \cite[Theorem 2.41]{thesis} (see also \cite[Section 3]{section3})
	\[
	E_\alpha(-t^{\alpha}\mathcal{L})=\frac{1}{2\pi i}\int_{H}e^{\gamma t}\gamma^{\alpha-1}(\gamma^{\alpha}+\mathcal{L})^{-1}d\gamma,\quad t\geqslant0,\quad 0<\alpha<1,
	\]
	where $H\subset\rho(-\mathcal{L})$ and $H$ is the Hankel's path of \cite[Formula (2.5)]{thesis}. Several properties of this operator can be found in \cite[Chapter 2]{thesis}.
\end{remark}

\begin{remark}
	Notice that taking the limit as $\alpha\to1$ in \eqref{asymtotic-heat} (in an informal way), we can see that the result in Theorem \ref{Main-heat} coincides with \cite[Corollary 7.1]{RR2020} up to some positive constant, which is what we expect for the classical propagator of the heat equation in the case of compact Lie groups. We also point out that the order $\alpha$ of the singular operator $\prescript{C}{}\partial_t^\alpha$ in equation \eqref{HeatTypeEquationG} is transferred to the time decay rate of the solution.    
\end{remark}

\subsection{Examples}\label{examples} Below we show several examples where the trace of the spectral projections is already known.   

\subsubsection*{Example 1} Let us consider the following heat type equation: 
\begin{align*}
	\left\{\begin{aligned}
		^{C}\partial_{t}^{\alpha}u(t,x)-\Delta_{sub}u(t,x)&=0, \quad t>0,\,\, x\in G, \,\, 0<\alpha\leqslant 1, \\
		u(t,x)|_{_{_{t=0}}}&=u_0(x),\quad u_0\in L^p(G),\quad 1<p\leqslant 2,
	\end{aligned}
	\right.
\end{align*}
where $\Delta_{sub}$ is the sub-Laplacian on a compact Lie group. By \cite{[35]}, it follows that the trace of the spectral projections $E_{(0,s)}(-\Delta_{sub})$ has the following asymptotic behavior:
\[
\tau\big(E_{(0,s)}(-\Delta_{sub})\big)\lesssim s^{Q/2},\quad s\to+\infty,
\]
where $Q$ is the Hausdorff dimension of $G$ with respect to the control distance generated by the sub-Laplacian. We note that if $\Delta_{sub}=\Delta_G$ is the Laplacian on $G$, then $Q=n$ is the topological dimension of $G.$ So, by Theorem \ref{Main-heat} we have the existence, uniqueness, the form, and the asymptotic behavior for the solution $u(t,x)$ as follows:
\[
\|u(t,\cdot)\|_{L^q(G)}\leqslant C_{\alpha,Q,p,q}t^{-\alpha Q/2\left(\frac{1}{p}-\frac{1}{q}\right)}\|u_0\|_{L^p(G)},\quad 2\leqslant q<+\infty,\quad \frac{2}{Q}>\frac{1}{p}-\frac{1}{q}. 
\]
In particular if we consider the case $\alpha=1/2$ we get the following integro-differential equation:
\begin{align*}
	\begin{aligned}
		\frac{1}{\Gamma(1/2)}\int_0^t (t-r)^{-1/2}u_r(r,x){\rm d}r-\Delta_{sub}u(t,x)&=0, \quad t>0,\,\, x\in G, \,\, 0<\alpha\leqslant 1, \\
		u(t,x)|_{_{_{t=0}}}&=u_0(x),\quad u_0\in L^p(G),\quad 1<p\leqslant 2,
	\end{aligned}
\end{align*}
whose solution can be given by $u(t,x)=E_{1/2}(-t^{1/2}\mathcal{L})u_0(x).$

\subsubsection*{Example 2} Let $G=\mathbb{T}^n$, $n \in \mathbb{N}.$ Consider the following heat type equation:
\begin{align}\label{ExampleTn}
	\begin{aligned}
		^{C}\partial_{t}^{\alpha}u(t,x)-\Delta u(t,x)&=0, \quad t>0,\,\, x\in \mathbb{T}^n, \\
		u(t,x)|_{_{_{t=0}}}&=u_0(x),\quad u_0\in L^p(\mathbb{T}^n),\quad 1<p\leqslant 2,
	\end{aligned}
\end{align}
where $\Delta$ is the Laplacian operator on $\mathbb{T}^n$. Since $\mathbb{T}^n$ is an abelian compact group, all its continuous unitary irreducible representations are one-dimensional and we can identify $\widehat{\mathbb{T}^n} \simeq \mathbb{Z}^n$ (see \cite{livropseudo} for a detailed approach of the Fourier theory on $\mathbb{T}^n$). Here, we have that $\sigma_{\Delta}(m)=|m|^2=\sum_{j=1}^n |m_j|^2$, for all $m\in\mathbb{Z}^n$. From Theorem \ref{Main-heat}, definition of the Lorentzian space and \cite[Prop. 2.9]{RR2020} we have that the solution of equation \eqref{ExampleTn} satisfies that
\[
\|u(t,\cdot) \|_{L^q(\mathbb{T}^n)} \leqslant \sup_{s>0}s\left(\sum_{\xi\in\mathbb{Z}^n:|E_\alpha(-t^\alpha |\xi|^2)|\geqslant s}1\right)^{\frac{1}{p}-\frac{1}{q}}\|u_0\|_{L^p(\mathbb{T}^n)}.
\]
Since, for each $t>0$, we have that $\sigma_{E_\alpha(-t^\alpha \Delta)}(m)=E_\alpha(-t^\alpha |m|^2),$ for all $m\in\mathbb{Z}^n$. We also have that  $\sigma_{E_\alpha(-t^\alpha \Delta)}(m) \to 0$, when $|m| \to \infty$ due to the uniform estimate \cite[P. 35]{page 35}, which provides that the above sum is finite and the existence of the supremum.

\section{Wave type equations with non-local differential operators}\label{wave-section}

We first recall a Sobolev space, which will be used in the results of this section. Thus, for $\beta\in\mathbb{R}$ and $1<p<+\infty$, the Sobolev space $\mathcal{H}_{\mathcal{L}}^{\beta,p}(G)$ is defined by
\[
\mathcal{H}_{\mathcal{L}}^{\beta,p}(G)=\big\{f:\,\, (I+\mathcal{L})^{\beta/2}f\in L^p(G)\big\}
\]
endowed with the norm
\[
\|f\|_{\mathcal{H}_{\mathcal{L}}^{\beta,p}(G)}=\|(I+\mathcal{L})^{\beta/2}f\|_{L^p(G)}.
\]
For $p=2$ we just use the standard notation $\mathcal{H}_{\mathcal{L}}^{\beta,2}(G)=\mathcal{H}_{\mathcal{L}}^{\beta}(G).$

\medskip Here we investigate the solution of the following equation, which interpolates between wave (without being wave,  $\alpha<2$) and heat types:  
\begin{equation}\label{WaveTypeEquationG}
	\left\{ \begin{aligned}
		^{C}\partial_{t}^{\alpha}u(t,x)+\mathcal{L}u(t,x)&=0, \quad t>0,\,\, x\in G, \\
		u(t,x)|_{_{_{t=0}}}&=u_0(x), \\
		\partial_t u(t,x)|_{_{_{t=0}}}&=u_1(x),
	\end{aligned}
	\right.
\end{equation}
where $\mathcal{L}$ is a positive linear left invariant operator (we always assume $\mathcal{L}: C^\infty(G) \to C^\infty(G)$ to be continuous), $1<\alpha<2$ and $u_0,u_1$ in some suitable Sobolev spaces. In the statements of the section we avoid the case $\alpha=1$ since it is already known. At this time, we are not able to use the same ideas of Section \ref{heat-section} on $L^p(G)-L^q(G)$ estimates $(1\leqslant p\leqslant 2\leqslant q<+\infty)$ since the propagators of the solution of equation \eqref{WaveTypeEquationG} have a different behaviour for the considered range of $\alpha$. In fact, we are loosing the complete monotonicity of the propagators, which is fundamental for that argument. Nevertheless, we can use the Fourier analysis of the group to prove existence of a solution on a Sobolev space in $L^2(G)$. 

\medskip Notice also that by using the Plancherel formula we can get for $\beta\in\mathbb{R}$ that  
\begin{align*}
	\|(I+\mathcal{L})^{\beta/2}u(t,\cdot)\|_{L^2(G)}^2 &=\sum_{[\xi]\in\widehat{G}}d_\xi \|\sigma_{(I+\mathcal{L})^{\beta/2}}(\xi)\widehat{u}(t,\xi)\|_{{\mathtt{HS}}}^2 \\
	&=\sum_{[\xi]\in\widehat{G}}d_\xi \sum_{i,j=1}^{d_\xi}(1+\mu_{i,\xi})^{\beta} |\widehat{u}(t,\xi)_{ij}|^2.
\end{align*}
Below we use the Borel functional calculus associated with the positive linear left invariant operator $\mathcal{L}$ and the two parametric Mittag-Leffler function $E_{\alpha,2}(-t^{\alpha}s)$ and $E_{\alpha}(-t^{\alpha}s)$ for $1<\alpha<2$ and $t,s\geqslant 0.$ The latter type of functions are holomorphic on the whole complex plane (entire function) \cite[Chapter 4]{mittag} and bounded for any $t,s\geqslant 0$ \cite[P. 35]{page 35}. %In this section we just consider those functions $E_{\alpha,2}(-t^{\alpha}s)$ which are $C^1[0,+\infty).$ 

We will also use the following propagator:
\[
E_{\alpha,2}(-t^{\alpha}\mathcal{L}) = \sum_{k=0}^{+\infty} \frac{(-t^{\alpha}\mathcal{L})^k}{\Gamma(\alpha k+2)},\quad t\geqslant0,\,\,1<\alpha<2.
\]

\begin{theorem}\label{Main-wave-a}
	Let $G$ be a compact Lie group, $1<\alpha<2$ and $\beta\in\mathbb{R}$. Let $\mathcal{L}$ be a positive linear left invariant operator on $G$. 
	\begin{enumerate}
		\item If $(u_0,u_1)\in \mathcal{H}_{\mathcal{L}}^{\beta}(G)\times \mathcal{H}_{\mathcal{L}}^{\beta}(G)$ then there exists a unique solution $u(t,\cdot)\in \mathcal{H}^{\beta+2}_{\mathcal{L}}(G)$ for any $t\in(0,+\infty)$ for the Cauchy problem \eqref{WaveTypeEquationG} given explicitly by
		\begin{equation}\label{solution-wave}
			u(t,x)=E_\alpha(-t^{\alpha}\mathcal{L})u_0(x)+tE_{\alpha,2}(-t^{\alpha}\mathcal{L})u_1(x),\quad x\in G,
		\end{equation}
		and we have 
\[
			\|u(t,\cdot)\|_{\mathcal{H}^{\beta+2}_{\mathcal{L}}(G)}\lesssim (1+t^{-\alpha})\|u_0\|_{\mathcal{H}_{\mathcal{L}}^{\beta}(G)}+t(1+t^{-\alpha})\|u_1\|_{\mathcal{H}_{\mathcal{L}}^{\beta}(G)}.
\]
		\item If $(u_0,u_1)\in \mathcal{H}^{\beta}_{\mathcal{L}}(G)\times \mathcal{H}^{\beta+2}_{\mathcal{L}}(G)$ then there exists a unique solution $u(t,\cdot)\in \mathcal{H}^{\beta+2}_{\mathcal{L}}(G)$ for any $t\in(0,+\infty)$ for the Cauchy problem \eqref{WaveTypeEquationG} given explicitly by \eqref{solution-wave}, and we have
\[
\|u(t,\cdot)\|_{\mathcal{H}^{\beta+2}_{\mathcal{L}}(G)}\lesssim (1+t^{-\alpha})\|u_0\|_{\mathcal{H}_{\mathcal{L}}^{\beta}(G)}+t\|u_1\|_{\mathcal{H}_{\mathcal{L}}^{\beta+2}(G)}.
			\]
	\item If $(u_0,u_1)\in \mathcal{H}^{\beta}_{\mathcal{L}}(G)\times \mathcal{H}_{\mathcal{L}}^{\beta+\frac{2(\alpha-1)}{\alpha}}(G)$ then there exists a unique solution $u(t,\cdot)\in \mathcal{H}^{\beta+2}_{\mathcal{L}}(G)$ for any $t\in(0,+\infty)$ for the Cauchy problem \eqref{WaveTypeEquationG} given explicitly by \eqref{solution-wave}, and we have
\[
\|u(t,\cdot)\|_{\mathcal{H}^{\beta+2}_{\mathcal{L}}(G)}\lesssim (1+t^{-\alpha})\|u_0\|_{\mathcal{H}_{\mathcal{L}}^{\beta}(G)}+t\|u_1\|_{\mathcal{H}_{\mathcal{L}}^{\beta}(G)}+\|u_1\|_{\mathcal{H}_{\mathcal{L}}^{\beta+\frac{2(\alpha-1)}{\alpha}}(G)}.
			\]
			
			\item If $(u_0,u_1)\in \mathcal{H}^{\beta+2}_{\mathcal{L}}(G)\times \mathcal{H}^{\beta}_{\mathcal{L}}(G)$ then there exists a unique solution $u(t,\cdot)\in \mathcal{H}^{\beta+2}_{\mathcal{L}}(G)$ for any $t\in(0,+\infty)$ for the Cauchy problem \eqref{WaveTypeEquationG} given explicitly by \eqref{solution-wave}, and we have
\[
\|u(t,\cdot)\|_{\mathcal{H}^{\beta+2}_{\mathcal{L}}(G)}\lesssim \|u_0\|_{\mathcal{H}_{\mathcal{L}}^{\beta+2}(G)}+t(1+t^{-\alpha})\|u_1\|_{\mathcal{H}_{\mathcal{L}}^{\beta}(G)}.
			\]
\item If $(u_0,u_1)\in \mathcal{H}^{\beta+2}_{\mathcal{L}}(G)\times \mathcal{H}^{\beta+2}_{\mathcal{L}}(G)$ then there exists a unique solution $u(t,\cdot)\in \mathcal{H}^{\beta+2}_{\mathcal{L}}(G)$ for any $t\in(0,+\infty)$ for the Cauchy problem \eqref{WaveTypeEquationG} given explicitly by \eqref{solution-wave}, and we have
\[
\|u(t,\cdot)\|_{\mathcal{H}^{\beta+2}_{\mathcal{L}}(G)}\lesssim \|u_0\|_{\mathcal{H}_{\mathcal{L}}^{\beta+2}(G)}+t\|u_1\|_{\mathcal{H}_{\mathcal{L}}^{\beta+2}(G)}.
	\]
	\item If $(u_0,u_1)\in \mathcal{H}^{\beta+2}_{\mathcal{L}}(G)\times \mathcal{H}_{\mathcal{L}}^{\beta+\frac{2(\alpha-1)}{\alpha}}(G)$ then there exists a unique solution $u(t,\cdot)\in \mathcal{H}^{\beta+2}_{\mathcal{L}}(G)$ for any $t\in(0,+\infty)$ for the Cauchy problem \eqref{WaveTypeEquationG} given explicitly by \eqref{solution-wave}, and we have
\[
\|u(t,\cdot)\|_{\mathcal{H}^{\beta+2}_{\mathcal{L}}(G)}\lesssim \|u_0\|_{\mathcal{H}_{\mathcal{L}}^{\beta+2}(G)}+t\|u_1\|_{\mathcal{H}_{\mathcal{L}}^{\beta}(G)}+\|u_1\|_{\mathcal{H}_{\mathcal{L}}^{\beta+\frac{2(\alpha-1)}{\alpha}}(G)}.
			\]
	\end{enumerate}
	Additionally, we can get that \begin{equation}\label{55}
		\|\partial_{t}u(t,\cdot)\|_{\mathcal{H}^{\beta}_{\mathcal{L}}(G)}\lesssim \left\{
		\begin{array}{rccl}
			t^{-1}\|u_0\|_{\mathcal{H}^{\beta}_{\mathcal{L}}(G)}+\|u_1\|_{\mathcal{H}^{\beta}_{\mathcal{L}}(G)},& u_{0},u_1\in \mathcal{H}^{\beta}_{\mathcal{L}}(G),\\
			\|u_0\|_{\mathcal{H}^{\beta+2/\alpha}_{\mathcal{L}}(G)}+\|u_1\|_{\mathcal{H}^{\beta}_{\mathcal{L}}(G)},& u_0\in \mathcal{H}^{\beta+2/\alpha}_{\mathcal{L}}(G),\, u_1\in \mathcal{H}^{\beta}_{\mathcal{L}}(G),
		\end{array}
		\right.
	\end{equation}
	for all $t\in(0,+\infty).$ All the above constants do not depend of $t>0.$
\end{theorem}
\begin{proof}
	We first note that by the spectral calculus it is enough to prove the theorem for $\beta=0$. Let $\widehat{u}(t,\xi)$, for $[\xi]\in \widehat{G}$, denote the group Fourier transform of $u$ with respect to the variable $x$. Thus
	\[
	\left\{ \begin{aligned}
		^{C}\partial_{t}^{\alpha}\widehat{u}(t,\xi)+\sigma_{\mathcal{L}}(\xi)\widehat{u}(t,\xi)&=0, \quad t>0, \\
		\widehat{u}(t,\xi)|_{_{_{t=0}}}&=\widehat{u_0}(\xi),  \\
		\partial_t \widehat{u}(t,\xi)|_{_{_{t=0}}}&=\widehat{u_1}(\xi).
	\end{aligned}
	\right.
	\]
	Now we have a system of scalar ODEs, which depend on the dimension $d_\xi$ of the representation $[\xi]$. So
	
	\[
	\left\{ \begin{aligned}
		^{C}\partial_{t}^{\alpha}\widehat{u}(t,\xi)_{ij}+\mu_{i,\xi}\widehat{u}(t,\xi)_{ij}&=0, \quad t>0, \\
		\widehat{u}(t,\xi)_{ij}|_{_{_{t=0}}}&=\widehat{u_0}(\xi)_{ij},  \\
		\partial_t \widehat{u}(t,\xi)_{ij}|_{_{_{t=0}}}&=\widehat{u_1}(\xi)_{ij},
	\end{aligned}
	\right.
	\]
	for any $i,j\in \{1,\ldots,d_\xi\}.$ Let us now apply the Laplace transform in the time-variable, to obtain
	\[
	\left\{ \begin{aligned}
		s^{\alpha}\widetilde{\widehat{u}}(s,\xi)_{ij}-s^{\alpha-1}\widehat{u_0}(\xi)_{ij}-s^{\alpha-2}\widehat{u_1}(\xi)_{ij}+\mu_{i,\xi} \widetilde{\widehat{u}}(s,\xi)_{ij}&=0, \quad s>0, \\
		\widehat{u}(t,\xi)_{ij}|_{_{_{t=0}}}&=\widehat{u_0}(\xi)_{ij},  \\
		\partial_t \widehat{u}(t,\xi)_{ij}|_{_{_{t=0}}}&=\widehat{u_1}(\xi)_{ij}.
	\end{aligned}
	\right.
	\]
	Thus, it follows that
	\[
	\widetilde{\widehat{u}}(s,\xi)_{ij}=\frac{s^{\alpha-1}}{s^{\alpha}+\mu_{i,\xi}}\widehat{u_0}(\xi)_{ij}+\frac{s^{\alpha-2}}{s^{\alpha}+\mu_{i,\xi}}\widehat{u_1}(\xi)_{ij}, \quad t>0, 
	\]
	and by the application of the inverse Laplace transform (see e.g. \cite[Theorem 2.1]{new-mittag-add}) we arrive at
	\begin{equation}\label{transform}
		\widehat{u}(t,\xi)_{ij}=E_{\alpha}(-\mu_{i,\xi}t^{\alpha})\widehat{u_0}(\xi)_{ij}+tE_{\alpha,2}(-\mu_{i,\xi}t^{\alpha})\widehat{u_1}(\xi)_{ij}.
	\end{equation}
	Now we apply the inverse Fourier transform on $G$ and get the explicit solution of the considered problem as follows: 
	\begin{align*}
		u(t,x)&=\sum_{[\xi]\in\widehat{G}}d_{\xi}\sum_{i,j=1}^{d_{\xi}}\widehat{u}(t,\xi)_{ij}\xi(x)_{ji} \\
		&=\sum_{[\xi]\in\widehat{G}}d_{\xi}\sum_{i,j=1}^{d_{\xi}}\left(E_{\alpha}(-\mu_{i,\xi}t^{\alpha})\widehat{u_0}(\xi)_{ij}+tE_{\alpha,2}(-\mu_{i,\xi}t^{\alpha})\widehat{u_1}(\xi)_{ij}\right)\xi(x)_{ji}\\
		&=\sum_{[\xi]\in\widehat{G}}d_{\xi}\sum_{i,j=1}^{d_{\xi}}E_{\alpha}(-\mu_{i,\xi}t^{\alpha})\widehat{u_0}(\xi)_{ij}\xi(x)_{ji}+t\sum_{[\xi]\in\widehat{G}}d_{\xi}\sum_{i,j=1}^{d_{\xi}}E_{\alpha,2}(-\mu_{i,\xi}t^{\alpha})\widehat{u_1}(\xi)_{ij}\xi(x)_{ji}\\
		&=E_\alpha(-t^\alpha \mathcal{L})u_0(x) + t E_{\alpha,2}(-t^\alpha\mathcal{L})u_1(x).
	\end{align*}
	
	On the other hand, by the equivalence in \eqref{transform} and estimate \cite[P. 35]{page 35} we obtain
	
	\begin{align}
		|\widehat{u}(t,\xi)_{ij}|&\leqslant  |E_{\alpha}(-\mu_{i,\xi}t^{\alpha})||\widehat{u_0}(\xi)_{ij}|+t|E_{\alpha,2}(-\mu_{i,\xi}t^{\alpha})||\widehat{u_1}(\xi)_{ij}| \nonumber\\
		&\leqslant C\frac{1}{1+\mu_{i,\xi}t^{\alpha}}|\widehat{u_0}(\xi)_{ij}|+C\frac{t}{1+\mu_{i,\xi}t^{\alpha}}|\widehat{u_1}(\xi)_{ij}|. \label{unomas}
	\end{align}
	
	Remember that some representations can have different eigenvalues which can be zero or positive. Taking it into account we first have 
\begin{align*}
\|u(t,\cdot)\|_{\mathcal{H}_{\mathcal{L}}^2(G)}^2&=\sum_{[\xi]\in\widehat{G}}d_\xi \sum_{i,j=1}^{d_\xi}(1+\mu_{i,\xi})^{2} |\widehat{u}(t,\xi)_{ij}|^2 \\ &\lesssim\sum_{[\xi]\in\widehat{G}}d_\xi \sum_{i,j=1}^{d_\xi}\frac{(1+\mu_{i,\xi})^{2}}{(1+\mu_{i,\xi}t^{\alpha})^{2}}\big(|\widehat{u_0}(\xi)_{ij}|^2+t^{2}|\widehat{u_1}(\xi)_{ij}|^2\big),
\end{align*}	
and therefore
\begin{align*}
\sum_{[\xi]\in\widehat{G}}d_\xi \sum_{i,j=1}^{d_\xi}\frac{(1+\mu_{i,\xi})^{2}}{(1+\mu_{i,\xi}t^{\alpha})^{2}}&|\widehat{u_0}(\xi)_{ij}|^2 \leqslant \sum_{[\xi]\in\widehat{G}}d_\xi \sum_{i,j=1}^{d_\xi}\left(1+\frac{\mu_{i,\xi}}{1+\mu_{i,\xi} t^{\alpha}}\right)^2|\widehat{u_0}(\xi)_{ij}|^2 \\
&\lesssim \left\{
\begin{array}{rccl}
&\displaystyle(1+t^{-\alpha})^2\sum_{[\xi]\in\widehat{G}}d_\xi \sum_{i,j=1}^{d_\xi}|\widehat{u_0}(\xi)_{ij}|^2=(1+t^{-\alpha})^2\|u_0\|_{L^2(G)}^2, \\
&\displaystyle\sum_{[\xi]\in\widehat{G}}d_\xi \sum_{i,j=1}^{d_\xi}(1+\mu_{i,\xi})^2|\widehat{u_0}(\xi)_{ij}|^2=\|u_0\|_{\mathcal{H}_{\mathcal{L}}^2(G)}^2. 
\end{array}   
\right. 
\end{align*}
The above estimates mean that we may not gain any decay due to the eigenvalues which are zero. We also have that 
\begin{align*}
\sum_{[\xi]\in\widehat{G}}d_\xi \sum_{i,j=1}^{d_\xi}&\frac{(1+\mu_{i,\xi})^{2}}{(1+\mu_{i,\xi}t^{\alpha})^{2}}t^2|\widehat{u_1}(\xi)_{ij}|^2\leqslant \sum_{[\xi]\in\widehat{G}}d_\xi \sum_{i,j=1}^{d_\xi}\left(1+\frac{\mu_{i,\xi}}{1+\mu_{i,\xi} t^{\alpha}}\right)^2 t^2|\widehat{u_1}(\xi)_{ij}|^2 \\
&\lesssim \left\{
\begin{array}{rccl}
&\displaystyle t^2(1+t^{-\alpha})^2\sum_{[\xi]\in\widehat{G}}d_\xi \sum_{i,j=1}^{d_\xi}|\widehat{u_1}(\xi)_{ij}|^2=t^2(1+t^{-\alpha})^2\|u_1\|_{L^2(G)}^2, \\
&\displaystyle t^2\sum_{[\xi]\in\widehat{G}}d_\xi \sum_{i,j=1}^{d_\xi}(1+\mu_{i,\xi})^2|\widehat{u_1}(\xi)_{ij}|^2=t^2 \|u_1\|_{\mathcal{H}_{\mathcal{L}}^2(G)}^2, \\
&\displaystyle \sum_{[\xi]\in\widehat{G}}d_\xi \sum_{i,j=1}^{d_\xi}\left(t+\mu_{i,\xi}\sup_{t>0}\frac{t}{1+\mu_{i,\xi} t^{\alpha}}\right)^2|\widehat{u_1}(\xi)_{ij}|^2.
\end{array}   
\right.
\end{align*}
One can check that the function $g(t)=\frac{t}{1+\mu_{i,\xi}t^{\alpha}}$ has a positive maximum at the point $t=\frac{1}{(\mu_{i,\xi}(\alpha-1))^{1/\alpha}}.$ Thus
	\begin{align*}
	\sum_{[\xi]\in\widehat{G}}d_\xi \sum_{i,j=1}^{d_\xi}&\left(t+\mu_{i,\xi}\sup_{t>0}\frac{t}{1+\mu_{i,\xi} t^{\alpha}}\right)^2|\widehat{u_1}(\xi)_{ij}|^2 \\
	&\leqslant t^2	\sum_{[\xi]\in\widehat{G}}d_\xi \sum_{i,j=1}^{d_\xi}|\widehat{u_1}(\xi)_{ij}|^2+C_{\alpha}  \sum_{[\xi]\in\widehat{G}}d_\xi \sum_{i,j=1}^{d_\xi}\mu_{i,\xi}^{\frac{2(\alpha-1)}{\alpha}}|\widehat{u_1}(\xi)_{ij}|^2 \nonumber\\
		&=t^2\|u_1\|_{L^2(G)}^2+ C_{\alpha}\|\mathcal{L}^{\frac{\alpha-1}{\alpha}}u_1\|_{L^2(G)}^2, %\label{cuatro4}
	\end{align*}
	for some positive constant $C_\alpha$ which depends only on $\alpha.$ This implies that 
\begin{align*}
\sum_{[\xi]\in\widehat{G}}d_\xi \sum_{i,j=1}^{d_\xi}\frac{(1+\mu_{i,\xi})^{2}}{(1+\mu_{i,\xi}t^{\alpha})^{2}}t^2|\widehat{u_1}(\xi)_{ij}|^2 &\lesssim \left\{
\begin{array}{rccl}
&t^2(1+t^{-\alpha})^2\|u_1\|_{L^2(G)}^2, \\
&t^2 \|u_1\|_{\mathcal{H}_{\mathcal{L}}^2(G)}^2, \\
&t^2\|u_1\|_{L^2(G)}^2+ \|\mathcal{L}^{\frac{\alpha-1}{\alpha}}u_1\|_{L^2(G)}^2.
\end{array}   
\right.
\end{align*}
Finally, combining the above estimates with 
\begin{align*}
\sum_{[\xi]\in\widehat{G}}d_\xi \sum_{i,j=1}^{d_\xi}\frac{(1+\mu_{i,\xi})^{2}}{(1+\mu_{i,\xi}t^{\alpha})^{2}}|\widehat{u_0}(\xi)_{ij}|^2 &\lesssim \left\{
\begin{array}{rccl}
&(1+t^{-\alpha})^2\|u_0\|_{L^2(G)}^2, \\
&\|u_0\|_{\mathcal{H}_{\mathcal{L}}^2(G)}^2, 
\end{array}   
\right. 
\end{align*}
we arrive at all the possible cases for $\|u(t,\cdot)\|_{\mathcal{H}_{\mathcal{L}}^2}^2$, keeping in mind for some of the cases that $\mathcal{H}_{\mathcal{L}}^{\beta+\frac{2(\alpha-1)}{\alpha}}(G)\subset \mathcal{H}^{\beta}_{\mathcal{L}}(G)$ for any $\beta\in\mathbb{R}.$

	Let us prove the last inequality \eqref{55} of this theorem. Notice that by equation \eqref{transform} and the property $\Gamma(\gamma+1)=\gamma\Gamma(\gamma)$ for $\Re(\gamma)>0$ we get
	\begin{align*}
		\partial_t \widehat{u}(t,\xi)_{ij}&=\sum_{k=1}^{+\infty}\frac{(-\mu_{i,\xi})^k \alpha kt^{\alpha k-1}}{\Gamma(\alpha k+1)}\widehat{u}_0(\xi)_{ij}+\sum_{k=0}^{+\infty}\frac{(-\mu_{i,\xi})^k (\alpha k+1)t^{\alpha k}}{\Gamma(\alpha k+2)}\widehat{u}_1(\xi)_{ij} \\
		&=-\mu_{i,\xi}t^{-1+\alpha}E_{\alpha,\alpha}(-\mu_{i,\xi}t^{\alpha})\widehat{u_0}(\xi)_{ij}+E_{\alpha}(-\mu_{i,\xi}t^{\alpha})\widehat{u_1}(\xi)_{ij},   
	\end{align*}
	where $E_{\alpha,\alpha}$ is defined in \eqref{bimittag}. This implies 
	\begin{align*}
		|\partial_t \widehat{u}(t,\xi)_{ij}|^2 &\lesssim t^{2(\alpha-1)}\mu_{i,\xi}^2|E_{\alpha,\alpha}(-\mu_{i,\xi}t^{\alpha})|^2|\widehat{u_0}(\xi)_{ij}|^2+|E_{\alpha}(-\mu_{i,\xi}t^{\alpha})|^2 |\widehat{u_1}(\xi)_{ij}|^2 \\
		&\lesssim \frac{t^{2(\alpha-1)}\mu_{i,\xi}^2}{1+\mu_{i,\xi}^2 t^{2\alpha}}|\widehat{u_0}(\xi)_{ij}|^2+\frac{1}{1+\mu_{i,\xi}^2 t^{2\alpha}} |\widehat{u_1}(\xi)_{ij}|^2 \\
		&\lesssim \frac{t^{2(\alpha-1)}\mu_{i,\xi}^2}{1+\mu_{i,\xi}^2 t^{2\alpha}}|\widehat{u_0}(\xi)_{ij}|^2+|\widehat{u_1}(\xi)_{ij}|^2 \\
		&\lesssim \left\{
		\begin{array}{rccl}
			t^{-2}|\widehat{u_0}(\xi)_{ij}|^2+|\widehat{u_1}(\xi)_{ij}|^2,\\
			\underbrace{\mu_{i,\xi}^2\,\displaystyle\sup_{t>0}\left(\frac{t^{2(\alpha-1)}}{1+\mu_{i,\xi}^2 t^{2\alpha}}\right)}_{\text{supremum at $t=(\alpha-1)^{1/2\alpha}/\mu_{i,\xi}^{1/\alpha}$}}|\widehat{u_0}(\xi)_{ij}|^2+|\widehat{u_1}(\xi)_{ij}|^2,
		\end{array} 
		\right. \\
		&\lesssim \left\{
		\begin{array}{rccl}
			t^{-2}|\widehat{u_0}(\xi)_{ij}|^2+|\widehat{u_1}(\xi)_{ij}|^2,\\
			\mu_{i,\xi}^{2/\alpha}|\widehat{u_0}(\xi)_{ij}|^2+|\widehat{u_1}(\xi)_{ij}|^2,
		\end{array}
		\right.
	\end{align*}
	in view of the estimate \cite[P. 35]{page 35}.
	Therefore, by the Plancherel's formula and the above inequality, it follows that
	\begin{align*}
		\|\partial_t u(t,\cdot)\|_{L^2(G)}^2&=\sum_{[\xi]\in\widehat{G}}d_\xi \sum_{i,j=1}^{d_\xi}|\partial_t\widehat{u}(t,\xi)_{ij}|^2 \\
		&\lesssim\left\{
		\begin{array}{rccl}
			\displaystyle\sum_{[\xi]\in\widehat{G}}d_\xi \sum_{i,j=1}^{d_\xi}\big(t^{-2}|\widehat{u_0}(\xi)_{ij}|^2+|\widehat{u_1}(\xi)_{ij}|^2\big),\\
			\displaystyle\sum_{[\xi]\in\widehat{G}}d_\xi \sum_{i,j=1}^{d_\xi}\big(\mu_{i,\xi}^{2/\alpha}|\widehat{u_0}(\xi)_{ij}|^2+|\widehat{u_1}(\xi)_{ij}|^2\big),
		\end{array} 
		\right.
		\\
		&\lesssim \left\{
		\begin{array}{rccl}
			t^{-2}\|u_0\|_{L^2(G)}^2+\|u_1\|_{L^2(G)}^2,&\quad u_{0},u_1\in L^2(G),\\
			\|u_0\|_{\mathcal{H}^{2/\alpha}_{\mathcal{L}}(G)}^2+\|u_1\|_{L^2(G)}^2,&\quad u_0\in \mathcal{H}^{2/\alpha}_{\mathcal{L}}(G),\, u_1\in L^2(G),
		\end{array}
		\right.
	\end{align*} 
	proving \eqref{55}.
\end{proof}

\begin{remark}
	In Theorem \ref{Main-wave-a}, we can see that the propagators are expressed by the Mittag-Leffler functions $E_\alpha(-t)$, $E_{\alpha,2}(-t)$ for $t>0$ and $1<\alpha<2$. For $E_\alpha(-t)$, it is known that it has finite zeros \cite{[16]} (see also \cite{[15]}). Also,  $E_{\alpha,2}(-t)$ has finite zeros in this range, see e.g. \cite[P. 4]{zeros2} or \cite{zeros}.
\end{remark}

\begin{remark}
	Notice that from inequality \eqref{unomas} we can also get
	\[
	\|u(t,\cdot)\|_{L^2(G)}\lesssim \|(1+t^{\alpha}\mathcal{L})^{-1}u_0\|_{L^2(G)}+t\|(1+t^{\alpha}\mathcal{L})^{-1}u_1\|_{L^2(G)},\quad t\in(0,T].
	\]
\end{remark}

	\section{Multi-term heat type equations}\label{multi-section}

In this section we treat the case of multi-term heat type equations. We study the following equation:  
\begin{equation}\label{Multi-HeatTypeEquationG}
	\left\{ \begin{aligned}
		\prescript{C}{}\partial_{t}^{\alpha_0}u(t,x)+\gamma_1\prescript{C}{}\partial_{t}^{\alpha_1}u(t,x)+\cdots+\gamma_m\prescript{C}{}\partial_{t}^{\alpha_m}u(t,x)+\mathcal{L}u(t,x)&=0,\,\,  \\
		u(t,x)|_{_{_{t=0}}}&=u_0(x),
	\end{aligned}
	\right.
\end{equation}
for $t>0$ and $x\in G$, where $\mathcal{L}$ is a positive linear left invariant operator on $G$ (we always assume $\mathcal{L}: C^\infty(G) \to C^\infty(G)$ to be continuous), $u_0$ will be taken in a suitable Sobolev space, $\gamma_i>0$ $(i=1,\ldots,m)$ and $0<\alpha_m<\alpha_{m-1}<\cdots<\alpha_1<\alpha_0\leqslant1$.

The solution of equation \eqref{Multi-HeatTypeEquationG} is connected with the so-called multivariate Mittag-Leffler function, see \cite{ML-defined,operational1}, for recent extensions see e.g. \cite{nuevo,new-mittag-add}. This function together with its Laplace transform is an important ingredient in our analysis in this section. 
\begin{definition}\label{multivariate-def}
	Let $\alpha_{i},\lambda\in\mathbb{R}$ $(i=1,\ldots,m)$ with $\alpha_{i}>0$. The  \textit{multivariate Mittag-Leffler function} is defined as (\cite{ML-defined})
	\begin{equation}\label{multivariateML}
		E_{(\alpha_1,\ldots,\alpha_m),\lambda}(w_1,\ldots,w_m)=\sum_{k_1=0}^{\infty}\cdots\sum_{k_m=0}^{\infty}\frac{(k_1+\cdots+k_m)!}{\Gamma(\alpha_1 k_1+\cdots+\alpha_m k_m+\lambda)}\frac{w_1^{k_1}}{k_{1}!}\cdots\frac{w_m^{k_m}}{k_{m}!},
	\end{equation}
	for any complex numbers $w_1,\ldots,w_m\in\mathbb{C}$.
\end{definition}
Notice that the function in \eqref{multivariateML} can be also associated to a special case of the well-known Lauricella functions. This function is absolutely and locally uniformly convergent for the given parameters.

\medskip Below we use a very useful estimate of the multivariate Mittag-Leffler function established in \cite[Lemma 3.2]{multi-estimate}. The only disadvantage of the latter estimate is that $0<t\leqslant T<+\infty$. This means that up to now, to the best of our knowledge, there is not an uniform estimate for the multivariate Mittag-Leffler function where the constant does not depend on $t$.  

\begin{theorem}
	Let $G$ be a compact Lie group and $\beta\in\mathbb{R}$. Suppose also that $\mathcal{L}$ is a positive linear left invariant operator on $G$. 
	\begin{enumerate}
		\item If $u_0\in \mathcal{H}^{\beta}_{\mathcal{L}}(G)$ then there exists a unique solution $u(t,\cdot)\in \mathcal{H}^{\beta+2}_{\mathcal{L}}(G)$ for any $t\in(0,T]$ for the Cauchy problem \eqref{Multi-HeatTypeEquationG} given explicitly by
		\begin{equation}
			u(t,x)=\sum_{k=0}^{m}t^{\alpha_0-\alpha_k}E_{(\alpha_0-\alpha_1,\ldots,\alpha_0-\alpha_m,\alpha_0),\alpha_0-\alpha_k+1}(-\gamma_1 t^{\alpha_0-\alpha_1},\ldots,-\gamma_m t^{\alpha_0-\alpha_m},-t^{\alpha_0}\mathcal{L}) u_0(x),\label{multi-solution}
		\end{equation}
		and we have 
		\[
		\|u(t,\cdot)\|_{\mathcal{H}_{\mathcal{L}}^{\beta+2}(G)}\leqslant C_{T,\alpha_0,\ldots,\alpha_m}\left(\sum_{k=0}^{m}\gamma_k t^{\alpha_0-\alpha_k}\right)(1+t^{-\alpha_0})\|u_0\|_{\mathcal{H}^{\beta}_{\mathcal{L}}(G)}. 
		\]
		\item If $u_0\in \mathcal{H}_{\mathcal{L}}^{\beta+2}(G)$ then there exists a unique solution $u(t,\cdot)\in \mathcal{H}^{\beta+2}_{\mathcal{L}}(G)$ for any $t\in(0,T]$ for the Cauchy problem \eqref{Multi-HeatTypeEquationG} given explicitly by \eqref{multi-solution}, and we have 
		\[
		\|u(t,\cdot)\|_{\mathcal{H}_{\mathcal{L}}^{\beta+2}(G)}\leqslant C_{T,\alpha_0,\ldots,\alpha_m}\left(\sum_{k=0}^{m}\gamma_k t^{\alpha_0-\alpha_k}\right)\|u_0\|_{\mathcal{H}^{\beta+2}_{\mathcal{L}}(G)}. 
		\] 
	\end{enumerate}
\end{theorem}
\begin{proof}
	Again by the spectral calculus it is enough to prove the theorem for $\beta=0$. By applying the Fourier transform on $G$ to equation \eqref{Multi-HeatTypeEquationG} we have
	\[
	\left\{ \begin{aligned}
		\prescript{C}{}\partial_{t}^{\alpha_0}\widehat{u}(t,\xi)+\gamma_1\prescript{C}{}\partial_{t}^{\alpha_1}\widehat{u}(t,\xi)+\cdots+\gamma_m\prescript{C}{}\partial_{t}^{\alpha_m}\widehat{u}(t,\xi)+\sigma_{P}(\xi)\widehat{u}(t,\xi)&=0, \quad t>0, \\
		\widehat{u}(t,\xi)|_{_{_{t=0}}}&=\widehat{u_0}(\xi).
	\end{aligned}
	\right.
	\]
	Thus we arrive into the system of scalar ODEs:
	\[
	\left\{ \begin{aligned}
		^{C}\partial_{t}^{\alpha_0}\widehat{u}(t,\xi)_{ij}+\gamma_1\prescript{C}{}\partial_{t}^{\alpha_1}\widehat{u}(t,\xi)_{ij}+\cdots+\gamma_m\prescript{C}{}\partial_{t}^{\alpha_m}\widehat{u}(t,\xi)_{ij}+\mu_{i,\xi}\widehat{u}(t,\xi)_{ij}&=0, \\
		\widehat{u}(t,\xi)_{ij}|_{_{_{t=0}}}&=\widehat{u_0}(\xi)_{ij},
	\end{aligned}
	\right.
	\]
	for any $i,j\in \{1,\ldots,d_\xi\}.$ We apply the Laplace transform in the time-variable, and get
	\[
	\left\{ \begin{aligned}
		s^{\alpha_0}\widetilde{\widehat{u}}(s,\xi)_{ij}-s^{\alpha_0-1}\widehat{u_0}(\xi)_{ij}+\gamma_1 s^{\alpha_1}\widetilde{\widehat{u}}(s,\xi)_{ij}-\gamma_1 s^{\alpha_1-1}\widehat{u_0}(\xi)_{ij}+&\cdots \\
		&\hspace{-9cm}\cdots+\gamma_m s^{\alpha_m}\widetilde{\widehat{u}}(s,\xi)_{ij}-\gamma_m s^{\alpha_m-1}\widehat{u_0}(\xi)_{ij}+\mu_{i,\xi}\widetilde{\widehat{u}}(s,\xi)_{ij}=0, \quad s>0, \\
		\widehat{u}(t,\xi)_{ij}|_{_{_{t=0}}}&=\widehat{u_0}(\xi)_{ij}.
	\end{aligned}
	\right.
	\]
	Therefore
	\[
	\widetilde{\widehat{u}}(s,\xi)_{ij}=\frac{s^{\alpha_0-1}+\gamma_1 s^{\alpha_1-1}+\cdots+\gamma_m s^{\alpha_m-1}}{s^{\alpha_0}+\gamma_1 s^{\alpha_1}+\cdots+\gamma_m s^{\alpha_m}+\mu_{i,\xi}}\widehat{u_0}(\xi)_{ij}, \quad s>0, 
	\]
	and by the application of the inverse Laplace transform (see e.g. \cite[Theorem 2.1]{new-mittag-add}) we get
	\begin{align*}
		\widehat{u}(t,\xi)_{ij}=\sum_{k=0}^{m}\gamma_k t^{\alpha_0-\alpha_k}&E_{(\alpha_0-\alpha_1,\ldots,\alpha_0-\alpha_m,\alpha_0),\alpha_0-\alpha_k+1}(-\gamma_1 t^{\alpha_0-\alpha_1},\ldots \\
		&\hspace{3cm}\ldots,-\gamma_m t^{\alpha_0-\alpha_m},-\mu_{i,\xi} t^{\alpha_0})\widehat{u_0}(\xi)_{ij},
	\end{align*}
	where $\gamma_0=1.$ By applying the inverse Fourier transform on $G$ to the above equality we obtain the desired representation of the solution. By the above equivalence and \cite[Lemma 3.2]{multi-estimate} we have
	
	\begin{align}
		|\widehat{u}(t,\xi)_{ij}|&\leqslant \sum_{k=0}^{m}\gamma_k t^{\alpha_0-\alpha_k}|E_{(\alpha_0-\alpha_1,\ldots,\alpha_0-\alpha_m,\alpha_0),\alpha_0-\alpha_k+1}(-\gamma_1 t^{\alpha_0-\alpha_1},\ldots \nonumber\\
		&\hspace{5cm}\ldots,-\gamma_m t^{\alpha_0-\alpha_m},-\mu_{i,\xi} t^{\alpha_0})||\widehat{u_0}(\xi)_{ij}| \nonumber\\
		&\leqslant C_{T,\alpha_0,\ldots,\alpha_m}\frac{|\widehat{u_0}(\xi)_{ij}|}{1+\mu_{i,\xi} t^{\alpha_0}}\sum_{k=0}^{m}\gamma_k t^{\alpha_0-\alpha_k},\quad 0<t\leqslant T.\label{new-estimate-l2}
	\end{align}
	Thus
	\begin{align*}
		\|u(t,\cdot)\|_{\mathcal{H}_{\mathcal{L}}^2}^2&=\sum_{[\xi]\in\widehat{G}}d_\xi \sum_{i,j=1}^{d_\xi}(1+\mu_{i,\xi})^2|\widehat{u}(s,\xi)_{ij}|^2 \\
		&\leqslant C_{T,\alpha_0,\ldots,\alpha_m}^{2}\left(\sum_{k=0}^{m}\gamma_k t^{\alpha_0-\alpha_k}\right)^2\sum_{[\xi]\in\widehat{G}}d_\xi \sum_{i,j=1}^{d_\xi}\frac{(1+\mu_{i,\xi})^2}{(1+\mu_{i,\xi} t^{\alpha_0})^2}|\widehat{u_0}(\xi)_{ij}|^2 \\
		&\leqslant C_{T,\alpha_0,\ldots,\alpha_m}^{2}\left(\sum_{k=0}^{m}\gamma_k t^{\alpha_0-\alpha_k}\right)^2\sum_{[\xi]\in\widehat{G}}d_\xi \sum_{i,j=1}^{d_\xi}\left(1+\frac{\mu_{i,\xi}}{1+\mu_{i,\xi} t^{\alpha_0}}\right)^2|\widehat{u_0}(\xi)_{ij}|^2 \\
		&\leqslant \left\{
		\begin{array}{rccl}
			\displaystyle C_{T,\alpha_0,\ldots,\alpha_m}^{2}\left(\sum_{k=0}^{m}\gamma_k t^{\alpha_0-\alpha_k}\right)^2 (1+t^{-\alpha_0})^2 \sum_{[\xi]\in\widehat{G}}d_\xi \sum_{i,j=1}^{d_\xi} |\widehat{u_0}(\xi)_{ij}|^2,\\
			C_{T,\alpha_0,\ldots,\alpha_m}^{2}\displaystyle\left(\sum_{k=0}^{m}\gamma_k t^{\alpha_0-\alpha_k}\right)^2\sum_{[\xi]\in\widehat{G}}d_\xi \sum_{i,j=1}^{d_\xi}(1+\mu_{i,\xi})^2 |\widehat{u_0}(\xi)_{ij}|^2,
		\end{array} \right. \\
		&\leqslant \left\{
		\begin{array}{rccl}
			\displaystyle C_{T,\alpha_0,\ldots,\alpha_m}^{2}\displaystyle\left(\sum_{k=0}^{m}\gamma_k t^{\alpha_0-\alpha_k}\right)^2 (1+t^{-\alpha_0})^2 \|u_0\|_{L^2(G)}^2,\\
			C_{T,\alpha_0,\ldots,\alpha_m}^{2}\displaystyle\left(\sum_{k=0}^{m}\gamma_k t^{\alpha_0-\alpha_k}\right)^2\|u_0\|_{\mathcal{H}_{\mathcal{L}}^2(G)}^2,
		\end{array}
		\right.
	\end{align*}
	completing the proof.
\end{proof}

\begin{remark}
	Notice that we will not study the multi-term wave type equation since for some of its propagator terms we can not use the estimate of the multivariate Mittag-Leffler function \cite[Lemma 3.2]{multi-estimate} to prove that the norm is bounded in the considered solution-space. By the same reason we are not yet able to give $L^p(G)-L^q(G)$ estimates for the solution. 
\end{remark}

\begin{remark}
	From the estimate \eqref{new-estimate-l2} we can get a better estimate for the $L^2$-norm of the solution of equation \eqref{Multi-HeatTypeEquationG}. In fact, we obtain
	\[
	\|u(t,\cdot)\|_{L^2(G)}\leqslant C_{T,\alpha_0,\ldots,\alpha_m}\sum_{k=0}^{m}\gamma_k t^{\alpha_0-\alpha_k}\|(1+t^{\alpha_0}\mathcal{L})^{-1}u_0\|_{L^2(G)},\quad t\in(0,T].
	\]
\end{remark}

\section{Acknowledgements}
The authors were supported by the FWO Odysseus 1 grant G.0H94.18N: Analysis and Partial Differential Equations and by the Methusalem programme of the Ghent University Special Research Fund (BOF) (Grant number 01M01021). MR is also supported by EPSRC grant EP/R003025/2 and FWO Senior Research Grant G011522N.


\begin{thebibliography}{00}

\bibitem{RR2020} R. Akylzhanov, M. Ruzhansky. $L^p-L^q$ multipliers on locally compact groups. J. Funct. Anal., 278(3), 2020, \#108324.

\bibitem{RR2020preprint} R. Akylzhanov, M. Ruzhansky. $L^p-L^q$ multipliers on locally compact groups. Preprint, 2018. https:
//arxiv.org/abs/1510.06321. 

\bibitem{para} M. Allen, L. Caffarelli, A. Vasseur. A parabolic problem with a fractional-time derivative. Arch. Ration. Mech. Anal., 221(2), (2016) 603--630.

\bibitem{BorelFunctional} W. Arveson. A Short Course on Spectral Theory, vol. 209, Springer Science \& Business Media, 2006. 

%\bibitem{propagator} E. Bazhlekova. The abstract Cauchy problem for the fractional evolution equation. Fract. Calc. Appl. Anal., 1(3), (1998), 255--270.

%\bibitem{phdthesis} E. Bazhlekova. Fractional evolution equations in Banach spaces. PhD thesis, Technische Universiteit Eindhoven, 2001.

\bibitem{nuevo} E. Bazhlekova, I. Bazhlekov. Identification of a space-dependent source term in a nonlocal problem for the general time-fractional diffusion equation. J. Comput. Appl. Math. 386, (2021), \# 113213.

\bibitem{new-mittag-add} E. Bazhlekova. Completely monotone multinomial Mittag-Leffler type functions and diffusion equations with multiple time-derivatives. Fract. Calc. Appl. Anal., 24(1), (2021), 88--111.

%\bibitem{von1} B. Blackadar. Operator Algebras: Theory of $C^*$-Algebras and Von Neumann Algebras.  Vol. 122, Springer Science \& Business Media, 2006.

\bibitem{compact1} N. Bourbaki. Elements of mathematics. Lie groups and Lie algebras. Addison--Wesley, 1975.

\bibitem{mittagalpha} E. Capelas de Oliveira, F. Mainardi, J. Vaz Jr. Models based on Mittag-Leffler functions for anomalous
relaxation in dielectrics. Eur. Phys. J. Spec. Top., 193, (2011), 161--171.

\bibitem{thesis} P. M. Carvalho-Neto. Fractional differential equations: a novel study of local and global solutions in Banach spaces. PhD thesis, Universidade de S$\tilde{a}$o Paulo, S$\tilde{a}$o Carlos, 2013.

\bibitem{otherdomain} Z. Q. Chen, M. M. Meerschaert, E. Nane. Space-time fractional diffusion on bounded domains. J. Math. Anal. Appl., 393(2), (2012), 479--488.

\bibitem{apli2} P. Cl\'ement, G. Gripenberg, S. O. Londen. Schauder estimates for equations with fractional derivatives. Trans. Amer. Math. Soc., 352(5), (2000), 2239--2260.

\bibitem{diethelm} K. Diethelm. The analysis of fractional differential equations: An application-oriented exposition using differential operators of Caputo type. Springer, Heidelberg, 2010.

\bibitem{von} J. Dixmier. Von Neumann Algebras. North-Holland Pub. Co, Amsterdam, New York, 1981.

\bibitem{cauchy} S. D. Eidelman, A. N. Kochubei. Cauchy problem for fractional diffusion equations. J. Differ. Equ., 199(2), (2004), 211--255.

\bibitem{heat-compact2} H. D. Fegan. The heat equation on a compact Lie group. Trans. Am. Math. Soc., 246, (1978), 339--357.

\bibitem{heat-fundamental} H. D. Fegan. The fundamental solution of the heat equation on a compact Lie group. J. Differential Geom., 18(4), (1983), 659--668.

\bibitem{wave-compact1} H. D. Fegan. Differential equations on Lie groups and tori the wave equations and Huygen's principle. Rocky Mountain J. Math., 14(3), (1984), 699--704.

\bibitem{FR16} V. Fischer, M. Ruzhansky. Quantization on nilpotent {L}ie groups. Progress in Mathematics, vol. 314, Birkh\"{a}user/Springer, 2016.

\bibitem{Fujita1} Y. Fujita. Integrodifferential equation which interpolates the heat equation and the wave equation. Osaka J. Math., 27, (1990), 309--321.

\bibitem{Fujita2} Y. Fujita. Integrodifferential equation which interpolates the heat equation and the wave equation II. Osaka J. Math., 27, (1990), 797--804.

\bibitem{wave-compact2} C. Garetto, M. Ruzhansky. Wave equation for sums of squares on compact Lie groups. J. Differ. Equ., 258(12), (2015), 4324--4347.

\bibitem{mittag} R. Gorenflo, A. A. Kilbas, F. Mainardi, S. V. Rogosin. Mittag-Leffler Functions, Related Topics and Applications, 2nd ed. Springer Monographs in Mathematics, Springer, New York, 2020.

\bibitem{ML-defined} R. Gorenflo, Y. Luchko. Operational method for solving generalized Abel integral equation of second kind. Integr. Transf. Spec. Func. 5 (1997), 47--58.

\bibitem{[15]} R. Gorenflo, F. Mainardi.  Fractional oscillations and Mittag-Leffler functions. In Proceedings of the International Workshop on the Recent Advances in Applied Mathematics, Kuwait University, Kuwait City, Kuwait, 1996.

\bibitem{memory1} M. E. Gurtin, A. C. Pipkin. A general theory of heat conduction with finite wave speeds. Arch. Rational Mech. Anal., 31, (1968), 40--50.

\bibitem{functionalcalculus} M. Haase. The Functional Calculus for Sectorial Operators. Oper. Theory Adv.
Appl., 169, Birkh\"auser, 2006.

\bibitem{operational1} S. B. Hadid, Y. Luchko. An operational method for solving fractional differential equations of an arbitrary real order. Panamer. Math. J., 6, (1996), 57--73.

\bibitem{zeros2}  J. Hanneken, B. Achar. An Alpha-Beta Phase Diagram Representation of the Zeros and Properties of the Mittag-Leffler Function. Advances in Mathematical Physics, \#421685, (2013).

\bibitem{zeros} J. W. Hanneken, D. M. Vaught, B.N. Achar.  Enumeration of the Real Zeros of the Mittag-Leffler Function $E_\alpha(z)$, $1 <\alpha< 2$. In: Sabatier, J., Agrawal, O.P., Machado, J.A.T. (eds) Advances in Fractional Calculus. Springer, Dordrecht.

\bibitem{uno1} A. Hanyga. Multidimensional solutions of space--time fractional diffusion equations. R. Soc. Lond. Proc. Ser. A Math. Phys. Eng. Sci., 458(2018), (2002), 429--450.

\bibitem{[35]} A. Hassannezhad, G. Kokarev. Sub-Laplacian eigenvalue bounds on sub-Riemannian manifolds. Ann. Sc. Norm. Super. Pisa Cl. Sci. (5), 16(4), (2016), 1049--1092.

\bibitem{homogeneous}  S. Helgason. Wave equations on homogeneous spaces, in: Lie Group Representations, III, College Park, Md., 1982/1983, in: Lecture Notes in Math., vol. 1077, Springer, Berlin, 1984, 254--287.

\bibitem{apli1} R. Hilfer. On fractional diffusion and continuous time random walks. Phys. A, 329, (2003), 35--40.

\bibitem{uno2} J. Kemppainen, J. Siljander, V. Vergara, R. Zacher. Decay estimates for time--fractional and other non--local in time subdiffusion equations in $\mathbb{R}^d$. Math. Ann., 366(3), (2016), 941--979.

\bibitem{uno3} J. Kemppainen, J. Siljander, R. Zacherc. Representation of solutions and large-time behavior for fully nonlocal diffusion equations. J. Differ. Equ., 263(1), (2017), 149--201.

\bibitem{kilbas} A. A. Kilbas, H. M. Srivastava, J. J. Trujillo. Theory and Applications of Fractional Differential Equations. North-Holland Mathematics Studies, vol. 204. Elsevier Science B.V., Amsterdam, 2006.

\bibitem{uno4} Y. C. Kim, K. A. Lee. Regularity results for fully nonlinear parabolic integro-differential operators. Math. Ann., 357(4), (2013), 1541--1576.

%\bibitem{KMP21} A. Kirilov, W. A. A. de Moraes, R. Paleari. Global analytic hypoellipticity for a class of evolution operators on {$\Bbb{T}^1\times \Bbb{S}^3$}. Journal of Differential Equations, 296, (2021), 699--723.

%\bibitem{KMR21} A. Kirilov, W. A. A. de Moraes, M. Ruzhansky. Global hypoellipticity and global solvability for vector fields on compact {L}ie groups. Journal of Functional Analysis, 280(2), (2021), 108806.

\bibitem{[51]} H. Kosaki.  Non-commutative Lorentz spaces associated with a semi-finite von Neumann algebra and applications.  Proc. Japan Acad. Ser. A Math. Sci.,  57(6), (1981), 303--306.

\bibitem{multi-estimate} Z. Li, Y. Liu, M. Yamamoto.  Initial-boundary value problems for multiterm time-fractional diffusion equations with positive constant coefficients. Appl. Math. Comp. 257 (2015), 381--397.

\bibitem{memory2} R. K. Miller. An integro differential equation for rigid heat conductors with memory. J. Math. Anal. Appl., 66, (1978),  313--332.


\bibitem{heisen} A. I. Nachman. The wave equation on the Heisenberg group. Commun. Partial Differ. Equ., 7(6), (1982), 675--714.

%\bibitem{propagatorruso} D. G. Orlovsky. Determination of the parameter of the differential equation of fractional order with the Caputo derivative in Hilbert space. J. Phys.: Conf. Ser., 1205, \# 012042, (2019).

\bibitem{palmieri} A. Palmieri. Semilinear wave equation on compact Lie groups. J. Pseudo-Differ. Oper. Appl. 12(43), (2021).

\bibitem{Pollard} H. Pollard. The completely monotonic character of the Mittag-Leffler function $E_a(-x)$. Bull. Amer. Math. Soc. 54, (1948), 1115--1116.

\bibitem{page 35} I. Podlubny. Fractional Differential Equations. Academic Press, San Diego, 1999.

\bibitem{Riesz} M. Riesz. L'int\'egrale de Riemann-Liouville et le problem de Cauchy. Acta Math., 81, (1949), 1--223.

\bibitem{example} B. Ross, S. G. Samko, E. Love. Functions that have no first order derivate might have fractional derivatives of all orders less than one. Real Anal. Exch. 20(2), (1994/5), 140--157.

\bibitem{livropseudo} M. Ruzhansky, V. Turunen. Pseudo-differential operators and symmetries. Background analysis and advanced topics. Pseudo-Differential Operators. Theory and Applications, vol.~2, Birkh\"{a}user Verlag, Basel, 2010. 

\bibitem{initial} K. Sakamoto, M. Yamamoto. Initial value/boundary value problems for fractional diffusion-wave equations and applications to some inverse problems. J. Math. Anal. Appl., 382(1), (2011), 426--447.

\bibitem{samko} S. G. Samko, A. A. Kilbas, O. I. Marichev. Fractional integrals and derivatives, translated from the 1987 Russian original, Gordon and Breach, Yverdon, 1993.

\bibitem{FractionalDiffusion} W. R. Schneider, W. Wyss. Fractional diffusion and wave equations. J. Math. Phys., 30 (1989), 134--144.

\bibitem{Mittag-bounded} T. Simon. Comparing Frechet and positive stable laws. Electron. J. Probab. 19, (2-14), 1--25.

\bibitem{symbol} N. Sugimoto, Y. Yamada, T. Kakutani. Torsional shock waves in a viscoelastic rod. Trans. ASME J. Appl. Mech., 51 (1984), 595--601.

%\bibitem{von2} M. Terp. $L^p$ Spaces Associated with Von Neumann Algebras. Copenhagen University, 1981.

\bibitem{heat-compact1} H. Urakawa. The heat equation on compact Lie group. Osaka J. Math., 12, (1975), 285--297.

\bibitem{uno5} V. Vergara, R. Zacher. Optimal decay estimates for time-fractional and other nonlocal subdiffusion equations via e nergy methods. SIAM J. Math. Anal., 47(1), (2015), 210--239.

\bibitem{section3} R-N. Wang, D-H. Chen, T-J. Xiao. Abstract fractional Cauchy problems with almost sectorial operators. J. Differ. Equ., 252(1), (2012), 202--235.

\bibitem{[16]} A. Wiman. Uber die Nullstellen der Funktionen $E_\alpha(x)$. Acta Math., 29(1), (1905), 217--234.

\bibitem{compact2} D. P. Zhelobenko. Compact Lie groups and their representations. Amer. Math. Soc., 1973, (Translated from Russian).


\end{thebibliography}
\end{document}